\documentclass[12pt,a4paper]{article}
\usepackage[T2A]{fontenc}       % shrifty
\usepackage{amsfonts}
\usepackage{latexsym}
\usepackage[reqno]{amsmath}
\usepackage{amsthm}
\usepackage{amsopn}
\usepackage{amssymb}
\usepackage[dvips]{color}
\usepackage{cite}
\usepackage[title]{appendix}
\usepackage[unicode, pdftex]{hyperref}% ствоення переходів у посиланнях всередені pdf
\hypersetup{
  colorlinks   = true, %Colours links instead of ugly boxes
  urlcolor     = blue, %Colour for external hyperlinks
  linkcolor    = blue, %Colour of internal links
  citecolor   = blue
}

\theoremstyle{plain}
\newtheorem{theorem}{$\indent$Theorem}
\newtheorem{corollary}{$\indent$Corollary}
\newtheorem{lemma}{$\indent$Lemma}

\theoremstyle{definition}
\newtheorem{definition}{$\indent$Definition}
\theoremstyle{remark}
\newtheorem{remark}{$\indent$Remark}
\newtheorem{example}{$\indent$Example}

\renewcommand \l[1]{\left#1}
\renewcommand \r[1]{\right#1}

\makeatletter
\def\blfootnote{\xdef\@thefnmark{}\@footnotetext}
\makeatother

\newenvironment{prfthm1}[1][Proof of Theorem 1]{\textit{#1.}\:\:}{\: \rule{0.7em}{0.7em}}
\newenvironment{prfthm2}[1][Proof of Theorem 2]{\textit{#1.}\:\:}{\: \rule{0.7em}{0.7em}}
\newenvironment{prfapthm}[1][Proof]{\textit{#1.}\:\:}{\: \rule{0.7em}{0.7em}}

\DeclareMathOperator{\re}{Re}
\DeclareMathOperator{\ExpV}{E}
\DeclareMathOperator{\Var}{Var}
\DeclareMathOperator{\Corr}{Corr}
\DeclareMathOperator{\Prob}{P}
\DeclareMathOperator{\diag}{diag}

\newcommand \OpLim[3]{\overset{#3}{\underset{#2}{#1}}}

\begin{document}

\author{A.V. Ivanov, N.N. Leonenko, I.V. Orlovskyi}
\title{On the Whittle estimator for linear random noise spectral density parameter in continuous-time nonlinear regression models}
\renewcommand\footnotemark{}
\blfootnote{\hspace*{-6mm}A.V. Ivanov \\
              National Technical University of Ukraine "Igor Sikorsky Kyiv Polytechnic Institute",\\
               Department of mathematical analysis and probability theory,\\
               Peremogi avenue 37, Kyiv, Ukraine \\
              Tel.: +380-44-2049740\\
              E-mail: alexntuu@gmail.com\smallskip  \\
           N.N. Leonenko \\
              Cardiff University, \\
              School of Mathematics,\\
              Senghennydd Road, Cardiff CF24 4AG, United Kingdom\\
              Tel.: +44(0)-29-2087-552\\
              E-mail: LeonenkoN@cardiff.ac.uk\smallskip\\
           I.V. Orlovskyi \\
              National Technical University of Ukraine "Igor Sikorsky Kyiv Polytechnic Institute"\\
              Department of mathematical analysis and probability theory,\\
              Peremogi avenue 37, Kyiv, Ukraine \\
              Tel.: +380-44-2049740\\
              E-mail: orlovskyi@matan.kpi.ua}
\date{}
\maketitle

\begin{abstract}
 A continuous-time nonlinear regression model with Lévy-driven linear noise process is considered. Sufficient conditions of consistency and asymptotic normality of the Whittle estimator for the parameter of the noise spectral density are obtained in the paper.\smallskip
\end{abstract}

\section{Introduction}\label{sec_Intro}

$\indent$The paper is focused on such an important aspect of the study of regression models with correlated observations as an estimation of random noise functional characteristics. When considering this problem the regression function unknown parameter becomes nuisance and complicates the analysis of noise. To neutralise its presence, we must estimate the parameter and then build estimators, say, of spectral density parameter of a stationary random noise using residuals, that is the difference between the values of the observed process and fitted regression function.

So, in the first step we employ the least squares estimator~(LSE) for unknown parameter of nonlinear regression, because of its relative simplicity. Asymptotic properties of the LSE in nonlinear regression model were studied by many authors. Numerous results on the subject can be found in monograph by Ivanov and Leonenko\cite{IvLeo_SAoRF_En}, Ivanov\cite{Iv_AToNR}.

In the second step we use the residual periodoram to estimate the unknown parameter of the noise spectral density using the Whittle-type contrast process~\cite{Wh_HTiTS,Wh_EaIiSTS}.

The results obtained at this time on the Whittle minimum contrast estimator~(MCE) form a developed theory that covers various mathematical models of stochastic processes and random fields. Some publications on the topic are Hannan~\cite{Han_MTS,Han_tAToLTSM}, Dunsmuir and Hannan~\cite{DunHan_VLTSM},  Guyon~\cite{Guy_PEfSPodDL}, Rosenblatt~\cite{Ros_SSaRF}, Fox and Taqqu~\cite{FoxTa_LSPoPE4SDSGTS}, Dahlhaus~\cite{Dah_EPE4SSP}, Heyde and Gay~\cite{HeGay_oAQLSP, HeGay_SPAaE4P&FwPLRD}, Giraitis and Surgailis~\cite{GiSur_CLT4QFiSDLVaiA}, Giraitis and Taqqu~\cite{GiTa_WE4FVnGTSwLM}, Gao et al~\cite{Gaoetal_PEoSPwLRD&I}, Gao~\cite{Gao_MLRDGPwAiCTFM}, Leonenko and Sakhno~\cite{LeoSa_oWE4SCoCPRP}, Bahamonde and Doukhan~\cite{BaDou_SEitPoMD}, Ginovyan and Sahakyan~\cite{GiSah_RE4CTLMwM}, Avram et al~\cite{AvLeoSa_oSTLTHUBLIetc}, Anh et al~\cite{AnhLeoSa_oCoMCE4FSP}, Bai et al~\cite{BaiGiTa_LT4QFoLDCTLP}, Ginovyan et al~\cite{GiSaTa_tTP4TM&OaiIiP}, Giraitis et al~\cite{GiTaTa_ANoQFoMD}.

In the article by Koul and Surgailis\cite{KoSur_ANoWEiLRMwLME} in the linear regression model the asymptotic properties of the Whittle estimator of strongly dependent random noise spectral density parameters were studied in a discrete-time setting.

In the paper by Ivanov and Prihod'ko\cite{IvPr_oWEoPoSDoRNiNRM} sufficient conditions on consistency and asymptotic normality of the Whittle estimator of the spectral density parameter of the Gaussian stationary random noise in continuous-time nonlinear regression model were obtained using residual periodogram. The current paper continues this research extending it to the case of the Lévy-driven linear random noise and more general classes of regression functions including trigonometric ones. We use the scheme of the proof in the case of Gaussian noise~\cite{IvPr_oWEoPoSDoRNiNRM} and some results of the papers~\cite{AvLeoSa_oSTLTHUBLIetc,AnhLeoSa_oCoMCE4FSP}. For linear random noise the proofs utilize essentially another types of limits theorems. In comparison with Gaussian case it leads to the use of special conditions on linear Lévy-driven random noise, new consistency and asymptotic normality conditions.

In the present publication continues-time model is considered. However, the results obtained can be also used for discrete time observations using the statements like Theorem 3 of Alodat and Olenko~\cite{AlOle_WCoWAFoLRDF} or Lemma 1 of Leonenko and Taufer~\cite{LeoTau_WCoFoSLMP2RtD}.
\section{Setting}\label{sec_Setting}

$\indent$Consider a regression model
 \begin{equation}
    X(t)=g(t,\,\alpha_0)+\varepsilon(t),\ t\ge 0,
  \label{c_n_reg_m}
 \end{equation}
where $g\;:\;(-\gamma,\,\infty)\times\mathcal{A}_\gamma\ \rightarrow\ \mathbb{R}$ is a continuous function, $\mathcal{A}\subset\mathbb{R}^q$ is an open convex set, $\mathcal{A}_\gamma=\bigcup\limits_{\|e\|\leq 1}\l(\mathcal{A}+\gamma e\r)$, $\gamma$ is some positive number, $\alpha_0\in\mathcal{A}$ is a true value of unknown parameter, and $\varepsilon$ is a random noise described below.

\begin{remark}
 The assumption about domain  $(-\gamma,\,\infty)$  for   function $g$ in $t$ is of technical nature and does not effect possible applications. This assumption makes it possible to formulate the condition \textbf{N$_2$}, which is used in the proof of Lemma \ref{lema_cnv_Dlta_2_0}.
\end{remark}

Throughout the paper $(\Omega,\,\mathcal{F},\,\Prob)$ denotes a complete probability space.

A Lévy process $L(t)$, $t\ge0$, is a stochastic process, with independent and stationary increments, continuous in probability, with sample-paths which are right-continuous with left limits (càdlàg) and $L(0)=0$. For a general treatment of Lévy processes we refer to Applebaum~\cite{App_LPaSC} and Sato~\cite{Sato_LPaIDP}.

Let $(a,\,b,\,\Pi)$ denote a characteristic triplet of the Lévy process $L(t)$, $t\ge0$, that is for all $t\ge0$
 \[
   \log \ExpV\exp\l\{\mathrm{i}zL(t)\r\}=t\kappa(z)
 \]
for all $z\in\mathbb{R}$, where
   \begin{equation}
    \kappa(z)=\mathrm{i}az-\frac12bz^2+\int\limits_{\mathbb{R}}\,\l(e^{\mathrm{i}zu}-1-\mathrm{i}z\tau(u)\r)\Pi(du),\ z\in\mathbb{R},
  \label{def_kappa}
 \end{equation}
where $a\in\mathbb{R}$, $b\ge0$, and
\[
  \tau(u)=\l\{\begin{array}{rl}
                  u, & |u|\le1; \\
                  \frac{u}{|u|}, & |u|>1.
              \end{array}
            \right.
\]
The Lévy measure $\Pi$ in \eqref{def_kappa} is a Radon measure on $\mathbb{R}\backslash\{0\}$ such that $\Pi(\{0\})=0$, and
 \[
   \int\limits_{\mathbb{R}}\,\min(1,\,u^2)\Pi(du)<\infty.
 \]

It is known that $L(t)$ has finite $p$-th moment for $p>0$ ($\ExpV|L(t)|^p<\infty$) if and only if
 \[
   \int\limits_{|u|\ge1}\,|u|^p\Pi(du)<\infty,
 \]
and $L(t)$ has finite $p$-th exponential moment for $p>0$ ($\ExpV\l[e^{pL(t)}\r]<\infty$) if and only if
 \begin{equation}
   \int\limits_{|u|\ge1}\,e^{pu}\Pi(du)<\infty,
  \label{bnd_int_e^luPi}
 \end{equation}
see, i.e., Sato~~\cite{Sato_LPaIDP}, Theorem 25.3.

If $L(t)$, $t\ge0$, is a Lévy process with characteristics $(a,\,b,\,\Pi)$, then the process $-L(t)$, $t\ge0$, is also a Lévy process with characteristics $(-a,\,b,\,\tilde{\Pi})$, where $\tilde{\Pi}(A)=\Pi(-A)$ for each Borel set $A$, modifying it to be càdlàg~\cite{AnhHeLeo_DMoLMPDbLP}.

We introduce a two-sided Lévy process $L(t)$, $t\in\mathbb{R}$, defined for $t<0$ to be equal an independent copy of $-L(-t)$.

Let $\hat{a}\,:\,\mathbb{R}\to\mathbb{R}_+$ be a measurable function. We consider the Lévy-driven continuous-time linear (or moving average) stochastic process
 \begin{equation}
    \varepsilon(t)=\int\limits_{\mathbb{R}}\, \hat{a}(t-s)dL(s),\ t\in\mathbb{R}.
  \label{LinRep_RmdnNse}
 \end{equation}
For causal process \eqref{LinRep_RmdnNse} $\hat{a}(t)=0,\ t<0$.

In the sequel we assume that
 \begin{equation}
    \hat{a}\in L_1(\mathbb{R})\cap L_2(\mathbb{R})\ \text{or}\ \hat{a}\in L_2(\mathbb{R}) \ \text{with}\ \ExpV L(1)=0.
  \label{hat_a_L_2&EL1=0}
 \end{equation}

Under the condition \eqref{hat_a_L_2&EL1=0} and
 \[
   \int\limits_{\mathbb{R}}\,u^2\Pi(du)<\infty,
 \]
the stochastic integral in \eqref{LinRep_RmdnNse} is well-defined in $L_2(\Omega)$ in the sense of stochastic integration introduced in Rajput and Rosinski \cite{RajRos_SRoIDP}.

The popular choices for the kernel in \eqref{LinRep_RmdnNse} are Gamma type kernels:
\begin{itemize}
  \item[$\cdot$] $\hat{a}(t)=t^\alpha e^{-\lambda t}\mathbb{I}_{[0,\,\infty)}(t)$, $\lambda>0$, $\alpha>-\frac12$;
  \item[$\cdot$] $\hat{a}(t)=e^{-\lambda t}\mathbb{I}_{[0,\,\infty)}(t)$, $\lambda>0$ (Ornstein-Uhlenbeck process);
  \item[$\cdot$] $\hat{a}(t)=e^{-\lambda|t|}$, $\lambda>0$ (well-balanced Ornstein-Uhlenbeck process).
\end{itemize}

\textbf{A$_1$.} The process $\varepsilon$ in \eqref{c_n_reg_m} is a measurable causal linear process of the form \eqref{LinRep_RmdnNse}, where a two-sides Lévy process $L$ is such that $\ExpV L(1)=0$, $\hat{a}\in L_1(\mathbb{R})\cap L_2(\mathbb{R})$. Moreover the Lévy measure $\Pi$ of $L(1)$ satisfies \eqref{bnd_int_e^luPi} for some $p>0$.

From the condition \textbf{A$_1$} it follows~\cite{AnhHeLeo_DMoLMPDbLP} for any $r\ge1$
 \begin{equation}
    \log \ExpV\exp\l\{\mathrm{i}\sum\limits_{j=1}^r\,z_j\varepsilon(t_j)\r\}= \int\limits_{\mathbb{R}}\,\kappa\l(\sum\limits_{j=1}^r\,z_j\hat{a}\l(t_j-s\r)\r)ds.
  \label{logEexp_isum_zEps}
 \end{equation}
In turn from \eqref{logEexp_isum_zEps} it can be seen that the stochastic process $\varepsilon$ is stationary in a strict sense.

Denote by
 \[
  \begin{aligned}
   m_r(t_1,\,\ldots,\,t_r)&=\ExpV \varepsilon(t_1)\ldots\varepsilon(t_r),\\
   c_r(t_1,\,\ldots,\,t_r)&=\mathrm{i}^{-r}\l.\dfrac{\partial^r}{\partial z_1\ldots\partial z_r}\,\log \ExpV\exp\l\{\mathrm{i}\sum\limits_{j=1}^r\,z_j\varepsilon(t_j)\r\}\,\r\vert_{z_1 = ... =z_r = 0}
  \end{aligned}
 \]
the moment and cumulant functions correspondingly of order $r,\ r\ge1$, of the process $\varepsilon$. Thus $m_2(t_1,\,t_2)=B(t_1-t_2)$, where
 \[
   B(t)=d_2\int\limits_{\mathbb{R}}\,\hat{a}(t+s)\hat{a}(s)ds,\ t\in\mathbb{R},
 \]
is a covariance function of $\varepsilon$, and the fourth moment function
 \begin{equation}
   \begin{aligned}
     m_4(t_1,\,t_2,\,t_3,\,t_4)=&c_4(t_1,\,t_2,\,t_3,\,t_4)+m_2(t_1,\,t_2)m_2(t_3,\,t_4)+\\
       &+m_2(t_1,\,t_3)m_2(t_2,\,t_4)+m_2(t_1,\,t_4)m_2(t_2,\,t_3).
   \end{aligned}
  \label{4ordr_mmnt_fn}
 \end{equation}

The explicit expression for cumulants of the stochastic process $\varepsilon$ can be obtained from \eqref{logEexp_isum_zEps} by direct calculations:
 \begin{equation}
    c_r(t_1,\,\ldots,\,t_r)=d_r\int\limits_{\mathbb{R}}\,\prod\limits_{j=1}^r\,\hat{a}\l(t_j-s\r)ds,
  \label{rep_c_r_by_hat-a}
 \end{equation}
where $d_r$ is the $r$-th cumulant of the random variable $L(1)$. In particular,
 \[
   d_2=\ExpV L^2(1)=-\kappa^{(2)}(0),\ \ \ d_4=\ExpV L^4(1) - 3\l(\ExpV L^2(1)\r)^2.
 \]

Under the condition \textbf{A$_1$}, the spectral densities of the stationary process $\varepsilon$ of all orders exist and can be obtained from \eqref{rep_c_r_by_hat-a} as
 \begin{equation}
    f_r(\lambda_1,\,\ldots,\,\lambda_{r-1})=(2\pi)^{-r+1}d_r\cdot a\l(-\sum\limits_{j=1}^{r-1}\,\lambda_j\r)\cdot \prod_{j=1}^{r-1}\,a(\lambda_j),
  \label{spec_den_H-ordr}
 \end{equation}
where $a\in L_2(\mathbb{R})$, $a(\lambda)=\int\limits_{\mathbb{R}}\,\hat{a}(t)e^{-\mathrm{i}\lambda t}dt$, $\lambda\in\mathbb{R}$, if complex-valued functions $f_r\in L_1\l(\mathbb{R}^{r-1}\r)$, $r>2$, see, e.g., \cite{AvLeoSa_oSTLTHUBLIetc} for definitions of the spectral densities of higher order $f_r,\ r\ge3$.

For $r=2$, we denote the spectral density of the second order by
 \[
   f(\lambda)=f_2(\lambda)=(2\pi)^{-1}d_2a(\lambda)a(-\lambda)=(2\pi)^{-1}d_2\l|a(\lambda)\r|^2.
 \]

\textbf{A$_2$.(i)} Spectral densities \eqref{spec_den_H-ordr} of all orders $f_r\in L_1(\mathbb{R}^{r-1})$, $r\ge2$;\\
\hspace*{11mm}\textbf{(ii)} $a(\lambda)=a\l(\lambda,\,\theta^{(1)}\r)$, $d_2=d_2\l(\theta^{(2)}\r)$, $\theta=\l(\theta^{(1)},\,\theta^{(2)}\r)\in\Theta_\tau$, $\Theta_\tau=\bigcup\limits_{\|e\|< 1}(\Theta+\tau e)$, $\tau>0$ is some number, $\Theta\subset\mathbb{R}^m$ is a bounded open convex set, that is $f(\lambda)=f(\lambda,\,\theta)$, $\theta\in\Theta_\tau$, and a true value of parameter $\theta_0\in\Theta$; \\
\hspace*{9mm}\textbf{(iii)} $f(\lambda,\,\theta)>0$, $(\lambda,\,\theta)\in\mathbb{R}\times\Theta^c$.

In the condition \textbf{A$_2$(ii)} above $\theta^{(1)}$ represents parameters of the kernel $\hat{a}$ in \eqref{LinRep_RmdnNse}, while $\theta^{(2)}$ represents parameters of Lévy process.

\begin{remark}
 The last part of the condition \textbf{A$_1$} is fully used in the proof of Lemma \ref{lema_AN_lin_fnctnl} and Theorem \ref{thm_AN4fncnl_appB} in Appendix \ref{app_LSE_AsymNorm}. The condition \textbf{A$_2$(i)} is fully used just in the proof of Lemma \ref{lema_AN_lin_fnctnl}. When we refer to these conditions in other places of the text we use them partially: see, for example, Lemma \ref{lema_cnv_J_T^eps}, where we need in the existence  of $f_4$ only.
\end{remark}

\begin{definition} \label{dfn_LSE}
 The least squares estimator (LSE) of the parameter $\alpha_0\in\mathcal{A}$ obtained by observations of the process $\l\{X(t),\ t\in[0,T]\r\}$ is said to be any random vector $\widehat{\alpha}_T=(\widehat{\alpha}_{1T},\,\ldots,\,\widehat{\alpha}_{qT})\in \mathcal{A}^c$ ($\mathcal{A}^c$ is the closure of $\mathcal{A}$), such that
  \[
    S_T\l(\widehat{\alpha}_T\r)= \min\limits_{\alpha\in\mathcal{A}^c}\,S_T(\alpha),\
    S_T(\alpha)=\int\limits_0^T\,\l(X(t)-g(t,\,\alpha)\r)^2dt.
  \]
\end{definition}

We consider the residual periodogram
 \[
   I_T(\lambda,\,\widehat{\alpha}_T)=(2\pi T)^{-1}\l|\int\limits_0^T\,\l(X(t)-g(t,\,\widehat{\alpha}_T)\r) e^{-\mathrm{i}t\lambda}dt\r|^2,\ \lambda\in \mathbb{R},
 \]
and the Whittle contrast field
 \begin{equation}
    U_T(\theta,\,\widehat{\alpha}_T)=\int\limits_{\mathbb{R}}\,\l(\log f(\lambda,\,\theta) +\dfrac{I_T\l(\lambda,\,\widehat{\alpha}_T\r)}{f(\lambda,\,\theta)}\r)w(\lambda)d\lambda,\ \theta\in\Theta^c,
  \label{Wh_ctrst_fn}
 \end{equation}
where $w(\lambda),\ \lambda\in\mathbb{R}$, is an even nonnegative bounded Lebesgue measurable function, for which the intgral \eqref{Wh_ctrst_fn} is well-defined. The existence of integral \eqref{Wh_ctrst_fn} follows from the condition \textbf{C$_4$} introduced below.

\begin{definition} \label{dfn_MCE}
 The minimum contrast estimator (MCE) of the unknown parameter $\theta_0\in\Theta$ is said to be any random vector $\widehat{\theta}_T=\l(\widehat{\theta}_{1T},...,\widehat{\theta}_{mT}\r)$ such that
  \[
    U_T\l(\widehat{\theta}_T,\,\widehat{\alpha}_T\r)=\min\limits_{\theta\in\Theta^c}\,U_T\l(\theta,\widehat{\alpha}_T\r).
  \]
\end{definition}

The minimum in the Definition \ref{dfn_MCE} is attained due to integral \eqref{Wh_ctrst_fn} continuity in $\theta\in\Theta^c$ as follows from the condition \textbf{C$_4$} introduced below.

\section{Consistency of the minimum contrast estimator}

$\indent$Suppose the function $g(t,\,\alpha)$ in \eqref{c_n_reg_m} is continuously differentiable with respect to $\alpha\in\mathcal{A}^c$ for any $t\ge0$, and its derivatives $g_i(t,\,\alpha)=\dfrac\partial{\partial\alpha_i}g(t,\,\alpha)$, $i=\overline{1,q}$, are locally integrable with respect to $t$. Let
 \[
   d_T(\alpha)=\diag\Bigl(d_{iT}(\alpha),\ i=\overline{1,q}\Bigr),\ d_{iT}^2(\alpha)=\int\limits_0^T\,g_i^2(t,\,\alpha)dt,
 \]
and $\underset{T\to\infty}\liminf\,T^{-\frac12}d_{iT}(\alpha)>0$, $i=\overline{1,q}$, $\alpha\in\mathcal{A}$.

Set
 \[
   \Phi_T(\alpha_1,\,\alpha_2)=\int\limits_0^T\, (g(t,\,\alpha_1)-g(t,\,\alpha_2))^2dt,\ \alpha_1,\,\alpha_2\in\mathcal{A}^c.
 \]

We assume that the following conditions are satisfied.

\textbf{C$_1$.} The LSE $\widehat{\alpha}_T$ is a weakly consistent estimator of $\alpha_0\in\mathcal{A}$ in the sense that
 \[
   T^{-\frac12}d_T(\alpha_0)\l(\widehat{\alpha}_T-\alpha_0\r)\ \overset{\Prob}\longrightarrow\ 0,\ \text{as}\ T\to\infty.
 \]

\textbf{C$_2$.} There exists a constant $c_0<\infty$ such that for any $\alpha_0\in\mathcal{A}$ and $T>T_0$, where $c_0$ and $T_0$ may depend on $\alpha_0$,
 \[
   \Phi_T(\alpha,\,\alpha_0)\leq c_0\|d_T(\alpha_0)\l(\alpha-\alpha_0\r)\|^2,\ \alpha\in\mathcal{A}^c.
 \]

The fulfillment of the conditions  C$_1$ and C$_2$ is discussed in more detail in Appendix \ref{app_LSE_cons}.

We need also in 3 more conditions.

\textbf{C$_3$.} $f(\lambda,\,\theta_1)\ne f(\lambda,\,\theta_2)$ on a set of positive Lebesgue measure once $\theta_1\ne\theta_2$, $\theta_1,\theta_2\in\Theta^c$.

\textbf{C$_4$.} The functions $w(\lambda)\log f(\lambda,\,\theta)$, $\dfrac{w(\lambda)}{f(\lambda,\,\theta)}$ are continuous with respect to $\theta\in\Theta^c$ almost everywhere in $\lambda\in\mathbb{R}$, and\\
\hspace*{14mm}\textbf{(i)} $w(\lambda)\l|\log f(\lambda,\,\theta)\r|\le Z_1(\lambda)$, $\theta\in\Theta^c$, almost everywhere in $\lambda\in\mathbb{R}$, and $Z_1(\cdot)\in L_1(\mathbb{R})$;\\
\hspace*{13mm}\textbf{(ii)} $\sup\limits_{\lambda\in\mathbb{R},\,\theta\in\Theta^c}\,\dfrac{w(\lambda)}{f(\lambda,\,\theta)} =c_1<\infty$.

\textbf{C$_5$.} There exists an even positive Lebesgue  measurable function $v(\lambda)$, $\lambda\in\mathbb{R}$,  such that\\
\hspace*{15mm}\textbf{(i)} $\dfrac{v(\lambda)}{f(\lambda,\,\theta)}$ is uniformly continuous in $(\lambda,\,\theta)\in\mathbb{R}\times\Theta^c$;\\
\hspace*{14mm}\textbf{(ii)} $\sup\limits_{\lambda\in\mathbb{R}}\,\dfrac{w(\lambda)}{v(\lambda)}<\infty$.

 \begin{theorem}\label{thm_MCE_cons}
  Under conditions  \textbf{A$_1$, A$_2$, C$_1$ -- C$_5$}\ \ $\widehat{\theta}_T\ \overset{\Prob}\longrightarrow\ \theta$, as $T\rightarrow\infty$.
 \end{theorem}

To prove the theorem we need some additional assertions.

 \begin{lemma}\label{lema_int_eps^2}
  Under condition \textbf{A$_1$}
   \[
     \nu^*_T=T^{-1}\int\limits_0^T\,\varepsilon^2(t)dt\ \overset{\Prob}\longrightarrow\ B(0),\ \ \text{as}\ \ T\to\infty.
   \]
 \end{lemma}

\begin{proof} For any $\rho>0$ by Chebyshev inequality and \eqref{4ordr_mmnt_fn}
  \[
   \begin{aligned}
     \Prob\l\{\l|\nu^*_T-B(0)\r|\ge\rho\r\}\le \rho^{-2}T^{-2}&\int\limits_0^T\int\limits_0^T\,c_4(t,t,s,s)dtds+ \\ &+2\rho^{-2}T^{-2}\int\limits_0^T\int\limits_0^T\,B^2(t-s)dtds=I_1+I_2.
   \end{aligned}
  \]

 From \textbf{A$_1$} it follows that $I_2=O(T^{-1})$. Using expression \eqref{rep_c_r_by_hat-a} for cumulants of the process $\varepsilon$ we get
  \[
   \begin{aligned}
     I_1&= d_4\rho^{-2}T^{-2}\int\limits_0^T\int\limits_0^T\int\limits_{\mathbb{R}}\,\hat{a}^2(t-u)\hat{a}^2(s-u)dudtds=\\
     &=d_4\rho^{-2}T^{-2}\int\limits_0^T\,\l(\int\limits_{\mathbb{R}}\,\hat{a}^2(t-u)\l(\int\limits_0^T\,\hat{a}^2(s-u)ds\r)du\r)dt \le d_4\rho^{-2}\l\|\hat{a}\r\|_2^4T^{-1},
   \end{aligned}
  \]
 where $\l\|\hat{a}\r\|_2=\l(\int\limits_{\mathbb{R}}\,\hat{a}^2(u)du\r)^{\frac12}$, that is $I_1=O(T^{-1})$ as well.
\end{proof}

$\indent$Let
 \[
  \begin{aligned}
   \mathrm{F}_T^{(k)}\l(u_1,\,\ldots,\,u_k\r)=\mathrm{F}_T^{(k)}\l(u_1\,\ldots,\,u_{k-1}\r)=&(2\pi)^{-(k-1)}T^{-1} \int\limits_{[0,T]^k}\, e^{\mathrm{i}\sum\limits_{j=1}^kt_j u_j}dt_1\ldots dt_k=\\
   =&(2\pi)^{-(k-1)}T^{-1}\prod\limits_{i=1}^k\, \dfrac{\sin\frac{Tu_j}2}{\frac{u_j}2},
  \end{aligned}
 \]
with $u_k=-\l(u_1+\ldots+u_{k-1}\r)$, $u_j\in\mathbb{R}$, $j=\overline{1,k}$.

The functions $\mathrm{F}_T^{(k)}\l(u_1,\ldots,u_k\r)$, $k\ge3$, are multidimensional analogues of the Fejér kernel, for $k=2$ we obtain the usual Fejér kernel.

The next statement bases on the results by R.~Bentkus~\cite{Ben_ANoEoSP,Ben_oEoEoSFoSP}, R.~Bentkus and R.~Rutkauskas~\cite{BenRut_oAo12Mo2OSE}.

 \begin{lemma}\label{lema_int_Fejer_krnl}
  Let function $G\l(u_1,\,\ldots,\,u_k\r)$, $u_k=-\l(u_1+\ldots+u_{k-1}\r)$ be bounded and continuous at the point $\l(u_1,\,\ldots,\,u_{k-1}\r)=(0,\,\ldots,\,0)$. Then
   \[
     \lim\limits_{T\rightarrow\infty}\,\int\limits_{\mathbb{R}^{k-1}}\,\mathrm{F}_T^k\l(u_1,\,\ldots,\,u_{k-1}\r) G\l(u_1,\,\ldots,\,u_k\r)du_1\ldots du_{k-1}=G(0,\,\ldots,\,0).
   \]
 \end{lemma}

We set
 \[
  \begin{aligned}
    g_T(\lambda,\,\alpha)&=\int\limits_0^T\,e^{-\mathrm{i}\lambda t}g(t,\,\alpha)dt,\ &s_T(\lambda,\,\alpha)&=g_T(\lambda,\,\alpha_0)-g_T(\lambda,\,\alpha),\\
    \varepsilon_T(\lambda)&=\int\limits_0^T\,e^{-\mathrm{i}\lambda t}\varepsilon(t)dt,\
    &I_T^{\varepsilon}(\lambda)&=(2\pi T)^{-1}\l|\varepsilon_T(\lambda)\r|^2,
  \end{aligned}
 \]
and write the residual periodogram in the form
 \[
   I_T\l(\lambda,\,\widehat{\alpha}_T\r)=I_T^{\varepsilon}(\lambda)+(\pi T)^{-1}\re\l\{\varepsilon_T(\lambda) \overline{s_T(\lambda,\,\widehat{\alpha}_T)}\r\} +(2\pi T)^{-1}\l|s_T(\lambda,\,\widehat{\alpha}_T)\r|^2.
 \]

Let $\varphi=\varphi (\lambda,\,\theta)$, $(\lambda,\,\theta)\in\mathbb{R}\times\Theta^c$, be an even Lebesgue measurable with respect to variable $\lambda$ for each fixed $\theta $ weight function. We have
  \[
   \begin{aligned}
     J_T(\varphi,\,\widehat{\alpha}_T)=&
      \int\limits_{\mathbb{R}}\,I_T(\lambda,\,\widehat{\alpha}_T)\varphi(\lambda,\,\theta)d\lambda= \int\limits_{\mathbb{R}}\,I_T^{\varepsilon}(\lambda)\varphi(\lambda,\,\theta)d\lambda +\\
     &+(\pi T)^{-1}\int\limits_{\mathbb{R}}\,\re\l\{\varepsilon_T(\lambda) \overline{s_T(\lambda,\,\widehat{\alpha}_T)}\r\} \varphi(\lambda,\,\theta)d\lambda+(2\pi T)^{-1} \int\limits_{\mathbb{R}}\,\l|s_T(\lambda,\,\widehat{\alpha}_T)\r|^2 \varphi(\lambda,\,\theta)d\lambda=\\
     &\ \ \ =J_T^{\varepsilon}(\varphi)+J_T^{(1)}(\varphi)+J_T^{(2)}(\varphi).
   \end{aligned}
  \]
Suppose
 \begin{equation}
    \varphi(\lambda,\,\theta)\ge0,\ \sup\limits_{\lambda\in\mathbb{R},\,\theta\in\Theta^c}\,\varphi(\lambda,\,\theta)= c(\varphi)<\infty   .
  \label{phi_pos&bnded}
 \end{equation}
Then by the Plancherel identity and condition \textbf{C$_2$}
 \[
  \begin{aligned}
     \l|J_T^{(1)}(\varphi)\r|&\le 2c(\varphi)\l((2\pi T)^{-1}\int\limits_{\mathbb{R}}\, |\varepsilon_T(\lambda)|^2d\lambda\r)^{\frac12}
     \l((2\pi T)^{-1}\int\limits_{\mathbb{R}}\,\l| s_T(\lambda,\,\widehat{\alpha}_T)\r|^2d\lambda\r)^{\frac12}=\\
     &=2c(\varphi)\l(\nu_T^*\r)^{\frac12}T^{-\frac12}\l(\Phi_T(\widehat{\alpha}_T,\,\alpha_0)\r)^{\frac12} \le2c_0^{\frac12}c(\varphi)\l(\nu_T^*\r)^{\frac12}\l\|T^{-\frac12}d_T(\alpha_0)\l(\widehat{\alpha}_T-\alpha_0\r)\r\|.
  \end{aligned}
 \]
Taking into account conditions \textbf{A$_1$, C$_1$, C$_2$} and the result of Lemma \ref{lema_int_eps^2} we obtain
 \begin{equation}
    \sup\limits_{\theta\in\Theta^c}\,\l|J_T^{(1)}(\varphi)\r|\ \overset{\Prob}\longrightarrow\ 0,\ \ \text{as}\ \ T\rightarrow\infty.
  \label{J_T^1phi_2_0}
 \end{equation}
On the other hand
 \[
   J_T^{(2)}(\varphi)\leq c(\varphi)T^{-1}\Phi_T(\alpha_0,\,\widehat{\alpha}_T)
   \le c_0c(\varphi)\l\|T^{-\frac12}d_T(\alpha_0)\l(\widehat{\alpha}_T-\alpha_0\r)\r\|^2,
 \]
and again, thanks to \textbf{C$_1$, C$_2$},
 \begin{equation}
    \sup\limits_{\theta\in\Theta^c}\,J_T^{(2)}(\varphi)\ \overset{\Prob}\longrightarrow\ 0,\ \ \text{as}\ \ T\rightarrow\infty.
  \label{J_T^2phi_2_0}
 \end{equation}

 \begin{lemma}\label{lema_cnv_J_T^eps}
  Suppose conditions \textbf{A$_1$, A$_2$} are fulfilled and the weight function $\varphi(\lambda,\,\theta)$ introduced above satisfies \eqref{phi_pos&bnded}. Then, as $T\to\infty$,
   \[
      J_T^{\varepsilon}(\varphi)\ \overset{\Prob}\longrightarrow\ J(\varphi)=\int\limits_{\mathbb{R}}\,f(\lambda,\,\theta_0) \varphi(\lambda,\,\theta)d\lambda,\ \theta\in\Theta^c.
   \]
 \end{lemma}

\begin{proof}
 The lemma in fact is an application of Lemma 2 in \cite{AnhHeLeo_DMoLMPDbLP} and Theorem 1 in \cite{AnhLeoSa_oCoMCE4FSP} reasoning to linear process \eqref{LinRep_RmdnNse}. It is sufficient to prove
  \[
     1)\ \ExpV J_T^\varepsilon(\varphi)\ \longrightarrow\ J(\varphi);\ \ \ 2)\ J_T^\varepsilon(\varphi)-\ExpV J_T^\varepsilon(\varphi)\ \overset{\Prob}\longrightarrow\ 0.
  \]

 Omitting parameters $\theta_0$, $\theta$ in some formulas below we derive
  \[
   \begin{aligned}
     \ExpV J_T^{\varepsilon}(\varphi)&=\int\limits_{\mathbb{R}}\,G_2(u)\mathrm{F}_T^{(2)}(u)du,\ \ G_2(u)=\int\limits_{\mathbb{R}}\,f(\lambda+u)\varphi(\lambda)d\lambda;\\
     T\Var J_T^\varepsilon(\varphi)&=2\pi\int\limits_{\mathbb{R}^3}\,G_4(u_1,\,u_2,\,u_3) \mathrm{F}_T^{(4)}(u_1,\,u_2,\,u_3) du_1du_2du_3,\\
     G_4(u_1&,\,u_2,\,u_3)=2\int\limits_{\mathbb{R}}\,f(\lambda+u_1)f(\lambda-u_3)\varphi(\lambda)\varphi(\lambda+u_1+u_2)d\lambda+\\
     &\hspace*{21mm}+\int\limits_{\mathbb{R}^2}\,f_4(\lambda+u_1,\,-\lambda+u_2,\,\mu+u_3)\varphi(\lambda) \varphi(\mu)d\lambda d\mu=\\
     &\hspace*{15mm}=2G_4^{(1)}(u_1,\,u_2,\,u_3)+G_4^{(2)}(u_1,\,u_2,\,u_3).
   \end{aligned}
  \]
 To apply Lemma \ref{lema_int_Fejer_krnl} we have to show that the functions $G_2(u)$, $u\in\mathbb{R}$; $G_4^{(1)}(\mathrm{u})$, $G_4^{(2)}(\mathrm{u})$, $\mathrm{u}=(u_1,\,u_2,\,u_3)\in\mathbb{R}^3$, are bounded and continuous at origins.

 Boundedness of $G_2$ follows from \eqref{phi_pos&bnded}. Thanks to \eqref{phi_pos&bnded}
  \[
    \underset{\mathrm{u}\in\mathbb{R}^3}\sup\,\l|G_4^{(1)}(\mathrm{u})\r|\le c^2(\varphi)\|f\|_2^2<\infty,\ \ \|f\|_2=\l(\int\limits_{\mathbb{R}}\,f^2(\lambda,\,\theta_0)d\lambda\r)^{\frac12}.
  \]
 On the other hand, by \eqref{spec_den_H-ordr}
  \[
   \begin{aligned}
     |G_4^{(2)}(u_1,\,u_2,\,u_3)|\le d_4&(2\pi)^{-3}\int\limits_{\mathbb{R}}\,\l|a(\lambda+u_1)a(-\lambda+u_2)\r|\varphi(\lambda)d\lambda \cdot\\
     &\cdot\int\limits_{\mathbb{R}}\,\l|a(\mu+u_3)a(-\mu-u_1-u_2-u_3)\r|\varphi(\mu)d\mu=d_4\cdot(2\pi)^{-3}\cdot I_3\cdot I_4,
   \end{aligned}
  \]
  \[
    I_3\le 2\pi c(\varphi)d_2^{-1}\int\limits_{\mathbb{R}}\,f(\lambda,\,\theta_0)d\lambda=2\pi c(\varphi)d_2^{-1}B(0).
  \]
 Integral $I_4$ admits the same upper bound. So,
  \[
    \underset{\mathrm{u}\in\mathbb{R}^3}\sup\,\l|G_4^{(2)}(\mathrm{u})\r|\le (2\pi)^{-1}\gamma_2c^2(\varphi)B^2(0),
  \]
 where $\gamma_2=\dfrac{d_4}{d_2^2}>0$ is the excess of $L(1)$ distribution, and functions $G_2$, $G_4^{(1)}$, $G_4^{(2)}$ are bounded. The continuity at origins of these functions follows from conditions of Lemma~\ref{lema_cnv_J_T^eps} as well.
\end{proof}

\begin{corollary}\label{cor_cnv_cntrst_fld}
 If $\varphi(\lambda,\,\theta)=\dfrac{w(\lambda)}{f(\lambda,\,\theta)}$, then under conditions \textbf{A$_1$, A$_2$, C$_1$, C$_2$} and \textbf{C$_4$}
  \[
    U_T(\theta,\,\widehat{\alpha}_T)\ \overset{\Prob}\longrightarrow\ U(\theta) =\int\limits_{\mathbb{R}}\,\l(\log f(\lambda,\,\theta)+ \dfrac{f(\lambda,\,\theta_0)}{f(\lambda,\,\theta)}\r)w(\lambda)d\lambda,\ \theta\in\Theta^c.
  \]
\end{corollary}

Consider the Whittle contrast function
 \[
   K(\theta_0,\,\theta)=U(\theta)-U(\theta_0)=\int\limits_{\mathbb{R}}\, \l(\dfrac{f(\lambda,\,\theta_0)}{f(\lambda,\,\theta)}-1- \log\dfrac{f(\lambda,\,\theta_0)}{f(\lambda,\,\theta)}\r) w(\lambda)d\lambda\ge 0,
 \]
with $K(\theta_0,\,\theta)=0$ if and only if $\theta=\theta_0$ due to \textbf{C$_3$}.

\begin{lemma}\label{lema_cnv_sup_dif_cntrst_fld}
 If the coditions \textbf{A$_1$, A$_2$, C$_1$, C$_2$, C$_4$} and \textbf{C$_5$} are satisfied, then
  \[
    \sup\limits_{\theta\in\Theta^c}\,\l|U_T(\theta,\,\widehat{\alpha}_T)- U(\theta)\r|\ \overset{\Prob}\longrightarrow\ 0,\ \ \text{as}\ \ T\to\infty.
  \]
\end{lemma}

\begin{proof}
 Let $\{\theta_j,\ j=\overline{1,N_{\delta}}\}$ be a $\delta$-net of the set $\Theta^c$. Then
  \[
   \begin{aligned}
     &\sup\limits_{\theta\in\Theta^c}\l|U_T(\theta,\,\widehat{\alpha}_T)-U(\theta)\r|\le \\
     &\ \ \ \le\sup\limits_{\|\theta_1-\theta_2\|\leq\delta} \l|U_T(\theta_1,\,\widehat{\alpha}_T)-U(\theta_1) -(U_T(\theta_2,\,\widehat{\alpha}_T)-U(\theta_2))\r|
     +\max\limits_{1\le j\le N_{\delta}} \l|U_T(\theta_j,\,\widehat{\alpha}_T)-U(\theta_j)\r|,
   \end{aligned}
  \]
 and for any $\rho\ge0$
  \[
   \begin{aligned}
     \Prob\l\{\sup\limits_{\theta\in\Theta^c}\,\l| U_T(\theta,\,\widehat{\alpha}_T)-U(\theta)\r|\ge\rho\r\}\leq P_1+P_2,
   \end{aligned}
  \]
 with
  \[
    P_2=\Prob\l\{\max\limits_{1\le j\le N_{\delta}}\,\l| U_T(\theta_j,\,\widehat{\alpha}_T)-U(\theta_j)\r| \ge\dfrac{\rho}2\r\}\ \to\ 0,\ \ \text{as}\ \ T\to\infty.
  \]
 by Corollary \ref{cor_cnv_cntrst_fld}. On the other hand,
  \begin{equation}
    \begin{aligned}
      P_1=\Prob&\l\{\sup\limits_{\|\theta_1-\theta_2\|\le\delta}\, \Bigl|U_T(\theta_1,\,\widehat{\alpha}_T)-U(\theta_1)- \l(U_T(\theta_2,\,\widehat{\alpha}_T)-U(\theta_2)\r) \Bigr|\ge\frac\rho2\r\}\le\\
      &\le \Prob\l\{\sup\limits_{\|\theta_1-\theta_2\|\le\delta}\, \l|\int\limits_{\mathbb{R}}\,I_T^{\varepsilon}(\lambda) \l(\dfrac{w(\lambda)}{f(\lambda,\,\theta_1)}- \dfrac{w(\lambda)}{f(\lambda,\,\theta_2)}\r)d\lambda\r|\r.+\\
      &\hspace*{9mm}+\sup\limits_{\|\theta_1-\theta_2\|\le\delta}\, \l|\int\limits_{\mathbb{R}}\,f(\lambda,\,\theta_0) \l(\dfrac{w(\lambda)}{f(\lambda,\,\theta_1)} -\dfrac{w(\lambda)}{f(\lambda,\,\theta_2)}\r)d\lambda\r|+\\
      &\hspace*{18mm}+2\l.\sup\limits_{\theta\in\Theta^c} \l|J_T^{(1)}\l(\dfrac{w}{f}\r)\r| +2\sup\limits_{\theta\in\Theta^c}\,J_T^{(2)}\l(\dfrac{w}{f}\r) \ge\dfrac{\rho}2\r\}.
    \end{aligned}
   \label{ineq_P_1}
  \end{equation}
 By the condition \textbf{C$_5$(i)}
  \[
    \sup\limits_{\|\theta_1-\theta_2\|\le\delta}\, \l|\int\limits_{\mathbb{R}}\, I_T^{\varepsilon}(\lambda) \l(\dfrac{w(\lambda)}{f(\lambda,\,\theta_1)} -\dfrac{w(\lambda)}{f(\lambda,\,\theta_2)}\r)d\lambda\r| \le\eta(\delta)\int\limits_{\mathbb{R}}\, I_T^{\varepsilon}(\lambda)\dfrac{w(\lambda)}{v(\lambda)}d\lambda,
  \]
 where
  \[
    \eta(\delta)=\sup\limits_{\lambda\in\mathbb{R},\, \|\theta_1-\theta_2\|\le\delta}\, \l|\dfrac{v(\lambda)}{f(\lambda,\,\theta_1)} -\dfrac{v(\lambda)}{f(\lambda,\,\theta_2)}\r|\ \to\ 0,\ \delta\to0.
  \]
 Since by Lemma \ref{lema_cnv_J_T^eps} and the condition \textbf{C$_5$(ii)}
  \[
    \int\limits_{\mathbb{R}}\,I_T^{\varepsilon}(\lambda)\dfrac{w(\lambda)}{v(\lambda)} d\lambda\overset{\Prob} \longrightarrow \int\limits_{\mathbb{R}}\,f(\lambda, \theta_0)\dfrac{w(\lambda)}{v(\lambda)} d\lambda,\ \ \text{as}\ \ T\rightarrow \infty,
  \]
 and the 2nd term under the probability sign in \eqref{ineq_P_1} by chosing $\delta$ can be made arbitrary small, then $P_1\to0$, as $T\to0$, taking into account that the 3rd and the 4th terms converge to zero in probability, thanks to \eqref{J_T^1phi_2_0} and \eqref{J_T^2phi_2_0}, if $\varphi=\dfrac wf$.
\end{proof}

\begin{prfthm1}
 By Definition \ref{dfn_MCE} for any $\rho>0$
  \[
   \begin{aligned}
     \Prob&\l\{\l\|\widehat{\theta}_T-\theta_0\r\|\ge\rho\r\} =\Prob\l\{\l\|\widehat{\theta}_T-\theta_0\r\|\ge\rho;\ U_T(\widehat{\theta}_T,\,\widehat{\alpha}_T)\le U_T(\theta_0,\,\widehat{\alpha}_T)\r\}\le\\
     &\le \Prob\l\{\inf\limits_{\|\theta-\theta_0\|\ge\rho}\, \l(U_T(\theta,\,\widehat{\alpha}_T) -U_T(\theta_0,\,\widehat{\alpha}_T)\r)\le 0\r\}=\\
     &= \Prob\l\{\inf\limits_{\|\theta-\theta_0\|\ge\rho}\, \Bigl[U_T(\theta,\,\widehat{\alpha}_T)-U(\theta) -(U_T(\theta_0,\,\widehat{\alpha}_T)-U(\theta_0)) +K(\theta_0,\theta)\Bigr]\leq 0\r\}\le\\
     &\le \Prob\l\{\inf\limits_{\|\theta-\theta_0\|\ge\rho}\, \Bigl[U_T(\theta,\,\widehat{\alpha}_T)-U(\theta) -(U_T(\theta_0,\,\widehat{\alpha}_T)-U(\theta_0))\Bigr] +\inf\limits_{\|\theta-\theta_0\|\ge\rho} K(\theta_0,\theta)\leq 0\r\}\le\\
     &\le \Prob\l\{\sup\limits_{\theta\in\Theta^c}\, \l|U_T(\theta,\,\widehat{\alpha}_T)-U(\theta)\r| +\l|U_T(\theta_0,\,\widehat{\alpha}_T)-U(\theta_0)\r| \ge\inf\limits_{\|\theta-\theta_0\|\ge\rho}\, K(\theta_0,\,\theta)\r\}\ \to\ 0,
   \end{aligned}
  \]
when $T\to\infty$ due to Lemma~\ref{lema_cnv_sup_dif_cntrst_fld} and the property of the contrast function $K$.
\end{prfthm1}

%-----------------------------------------------------------
\section{Asymptotic normality of minimum contrast estimator}

$\indent$The first three conditions relate to properties of the regression function $g(t,\,\alpha)$ and the LSE $\widehat{\alpha}_T$. They are commented in Appendix \ref{app_LSE_AsymNorm}.

\textbf{N$_1$.} The normed LSE $d_T(\alpha_0)\l(\widehat{\alpha}_T-\alpha_0\r)$ is asymptotically,\ \  as $T\to\infty$, normal $N(0,\,\Sigma_{_{LSE}})$, $\Sigma_{_{LSE}}=\l(\Sigma_{_{LSE}}^{ij}\r)_{i,j=1}^q$.

Let us
 \[
   g'(t,\,\alpha)=\dfrac{\partial}{\partial t}g(t,\,\alpha);\ \ \ \Phi'_T(\alpha_1,\,\alpha_2) =\int\limits_0^T\,\l(g'(t,\,\alpha_1)-g'(t,\,\alpha_2)\r)^2dt,\ \alpha_1,\,\alpha_2\in\mathcal{A}^c.
 \]

\textbf{N$_2$.} The function $g(t,\,\alpha)$ is continuously differentiable with respect to $t\ge0$ for any $\alpha\in\mathcal{A}^c$ and for any $\alpha_0\in\mathcal{A}$, and $T>T_0$ there exists a constant $c_0'$ ($T_0$ and $c'_0$ may depend on $\alpha_0$) such that
 \[
   \Phi_T'(\alpha,\,\alpha_0) \le c_0'\Bigl\|d_T(\alpha_0)\l(\alpha-\alpha_0\r)\Bigr\|^2,\ \alpha\in\mathcal{A}^c.
 \]

Let
 \[
   g_{il}(t,\,\alpha)=\dfrac{\partial^2}{\partial\alpha_i \partial\alpha_l}g(t,\,\alpha),\ \ d_{il,T}^2(\alpha)=\int\limits_0^T\,g_{il}^2(t,\,\alpha)dt,\ \ i,l=\overline{1,q},\ \ v(r)=\l\{x\in\mathbb{R}^q\,:\,\|x\|<r\r\},\ r>0.
 \]

\textbf{N$_3$.} The function $g(t,\,\alpha)$ is twice continuously differentiable  with respect to $\alpha\in\mathcal{A}^c$ for any $t\ge0$, and for any $R\ge0$ and all sufficiently large $T$ ($T>T_0(R)$)\\
\hspace*{15mm}\textbf{(i)} $d_{iT}^{-1}(\alpha_0)\sup\limits_{t\in[0,T],\,u\in v^c(R)}\, \l|g_i\l(t,\,\alpha_0+d_T^{-1}(\alpha_0)u\r)\r|\le c^i(R)T^{-\frac12}$, $i=\overline{1,q}$;\\
\hspace*{14mm}\textbf{(ii)} $d_{il,T}^{-1}(\alpha_0)\sup\limits_{t\in[0,T],\,u\in v^c(R)}\, \l|g_{il}\l(t,\,\alpha_0+d_T^{-1}(\alpha_0)u\r)\r|\le c^{il}(R)T^{-\frac12}$, $i,l=\overline{1,q}$;\\
\hspace*{12mm}\textbf{(iii)} $d_{iT}^{-1}(\alpha_0)d_{lT}^{-1}(\alpha_0) d_{il,T}(\alpha_0)\le \tilde{c}^{il}T^{-\frac12}$, $i,l=\overline{1,q}$,\\
with positive constants $c^i$, $c^{il}$, $\tilde{c}^{il}$, possibly, depending on $\alpha_0$.

We assume also that the function $f(\lambda,\,\theta)$ is twice differentiable with respect to $\theta\in\Theta^c$ for any $\lambda\in\mathbb{R}$.

Set
 \[
   f_i(\lambda,\,\theta)=\dfrac{\partial}{\partial\theta_i} f(\lambda,\,\theta),\ \ \
   f_{ij}(\lambda,\,\theta)=\dfrac{\partial^2}{\partial\theta_i \partial\theta_j}f(\lambda,\,\theta),
 \]
and introduce the following conditions.

\textbf{N$_4$. (i)} For any $\theta\in\Theta^c$ the functions $\varphi_i(\lambda)=\dfrac{f_i(\lambda,\,\theta)}{f^2(\lambda,\,\theta)}w(\lambda)$, $\lambda\in\mathbb{R}$, $i=\overline{1,m}$, possess the following properties:\\
\hspace*{18mm}\textbf{1)} $\varphi_i\in L_\infty(\mathbb{R})\cap L_1(\mathbb{R})$;\\
\hspace*{18mm}\textbf{2)} $\OpLim{\Var}{-\infty}{+\infty}\,\varphi_i<\infty$;\\
\hspace*{18mm}\textbf{3)} $\underset{\eta\to1}\lim\,\underset{\lambda\in\mathbb{R}}\sup\, \l|\varphi_i(\eta\lambda)-\varphi_i(\lambda)\r|=0$ ;\\
\hspace*{18mm}\textbf{4)} $\varphi_i$ are differentiable and $\varphi'_i$ are uniformly continuous on $\mathbb{R}$.\\
\hspace*{12mm}\textbf{(ii)} $\dfrac{|f_i(\lambda,\,\theta)|}{f(\lambda,\,\theta)}w(\lambda) \le Z_2(\lambda)$, $\theta\in\Theta$, $i=\overline{1,m}$, almost everywhere in $\lambda\in\mathbb{R}$ and $Z_2(\cdot)\in L_1(\mathbb{R})$.\\
\hspace*{11mm}\textbf{(iii)} The functions $\dfrac{f_i(\lambda,\,\theta) f_j(\lambda,\,\theta)}{f^2(\lambda,\,\theta)}w(\lambda)$, $\dfrac{f_{ij}(\lambda,\,\theta)}{f(\lambda,\,\theta)}w(\lambda)$ are continuous with respect to $\theta\in\Theta^c$ for each $\lambda\in\mathbb{R}$ and
 \[
   \dfrac{f_i^2(\lambda,\,\theta)}{f^2(\lambda,\,\theta)}w(\lambda) +\dfrac{|f_{ij}(\lambda,\,\theta)|}{f(\lambda,\,\theta)}w(\lambda)\le a_{ij}(\lambda),\ \lambda\in\mathbb{R},\ \theta\in\Theta^c,
 \]
where $a_{ij}(\cdot)\in L_1(\mathbb{R})$, $i,j=\overline{1,m}$.

\textbf{N$_5$.} \textbf{(i)} $\dfrac{f_i^2(\lambda,\,\theta)}{f^3(\lambda,\,\theta)}w(\lambda)$, $\dfrac{f_{ij}(\lambda,\,\theta)}{f^2(\lambda,\,\theta)}w(\lambda)$, $i,j=\overline{1,m}$, are bounded functions in $(\lambda,\,\theta)\in\mathbb{R}\times\Theta^c$;\\
\hspace*{12mm}\textbf{(ii)} There exists an even positive Lebesgue measurable function $v(\lambda),\ \lambda\in\mathbb{R}$, such that the functions $\dfrac{f_i(\lambda,\,\theta) f_j(\lambda,\,\theta)}{f^3(\lambda,\,\theta)}v(\lambda)$, $\dfrac{f_{ij}(\lambda,\,\theta)}{f^2(\lambda,\,\theta)}v(\lambda)$, $i,j=\overline{1,m}$, are uniformly continuous in $(\lambda,\,\theta)\in\mathbb{R}\times\Theta^c$;\\
\hspace*{11mm}\textbf{(iii)} $\underset{\lambda\in\mathbb{R}}\sup\,\dfrac{w(\lambda)}{v(\lambda)}<\infty$.

Conditions \textbf{N$_5$(iii)} and \textbf{C$_5$(ii)} look the same, however the function $v$ in these conditions must satisfy different conditions \textbf{N$_5$(ii)} and \textbf{C$_5$(i)}, and therefore, generally speaking, the functions $v$ in these two conditions can be different.

The next three matrices appear in the formulation of Theorem 2:
 \[
   \begin{aligned}
     W_1(\theta)&=\int\limits_{\mathbb{R}}\,\nabla_\theta\log f(\lambda,\,\theta)\nabla_\theta'\log f(\lambda,\,\theta)w(\lambda)d\lambda,\\
     W_2(\theta)&=4\pi\int\limits_{\mathbb{R}}\,\nabla_\theta\log f(\lambda,\,\theta)\nabla_\theta'\log f(\lambda,\,\theta) w^2(\lambda)d\lambda,\\
     V(\theta)&=\gamma_2\int\limits_{\mathbb{R}}\,\nabla_\theta\log f(\lambda,\,\theta)w(\lambda)d\lambda \int\limits_{\mathbb{R}}\,\nabla_\theta'\log f(\lambda,\theta)w(\lambda)d\lambda,
   \end{aligned}
 \]
where $\nabla_\theta$ is a column vector-gradient, $\nabla_\theta'$ is a row vector-gradient.

\textbf{N$_6$.} Matrices $W_1(\theta)$ and $W_2(\theta)$ are positive definite for $\theta\in\Theta$.

\begin{theorem}\label{thm_MCE_asym_norm}
 Under conditions \textbf{A$_1$, A$_2$, C$_1$ -- C$_5$} and \textbf{N$_1$ -- N$_6$} the normed MCE\ \  $T^{\frac12}(\widehat{\theta}_T-\theta_0)$ is asymptotically, as $T\to\infty$, normal with zero mean and covariance matrix
  \begin{equation}
      W(\theta)=W_1^{-1}(\theta_0)\l(W_2(\theta_0)+V(\theta_0)\r) W_1^{-1}(\theta_0).
   \label{cov_mtrx_W}
  \end{equation}
\end{theorem}

The proof of the theorem is preceded by several lemmas. The next statement is Theorem 5.1~\cite{AvLeoSa_oSTLTHUBLIetc} formulated in a form convenient to us.
\begin{lemma}\label{lema_AN_lin_fnctnl}
 Let the stochastic process\ \ $\varepsilon$\ \  satisfies \textbf{A$_1$, A$_2$},\ \ spectral density $f\in L_p(\mathbb{R})$,\ \ a function $b\in L_q(\mathbb{R})\bigcap L_1(\mathbb{R})$, where $\dfrac1{p}+\dfrac1{q}=\dfrac12$. Let
   \begin{equation}
     \hat{b}(t)=\int\limits_{\mathbb{R}}\,e^{i\lambda t}b(\lambda)d\lambda
   \label{def_Ftrfrm_b}
  \end{equation}
 and
  \begin{equation}
     Q_T=\int\limits_0^T\int\limits_0^T\, \l(\varepsilon(t)\varepsilon(s) -B(t-s)\r)\hat{b}(t-s)dtds.
   \label{def_fnctinl_Q_T}
  \end{equation}
 Then the central limit theorem holds:
  \[
    T^{-\frac12}Q_T\ \Rightarrow\ N(0,\,\sigma^2),\ \ \text{as}\ \ T\rightarrow\infty,
  \]
 where $"\Rightarrow"$ means convergence in distributions,
  \begin{equation}
     \sigma^2=16\pi^3\int\limits_{\mathbb{R}}\,b^2(\lambda) f^2(\lambda)d\lambda+\gamma_2\l(2\pi\int\limits_{\mathbb{R}}\, b(\lambda)f(\lambda)d\lambda\r)^2.
   \label{cov_mtrx_sigma^2}
  \end{equation}
 In particular, the statement is true for $p=2$ and $q=\infty$.
\end{lemma}

Alternative form of Lemma \ref{lema_AN_lin_fnctnl} is given in Bai et al.~\cite{BaiGiTa_LT4QFoLDCTLP}. We formulate their Theorem 2.1 in the form convenient to us.

\begin{lemma}\label{lema_AN_lin_fnctnl_Bai}
 Let the stochastic process $\varepsilon$  be such that $\ExpV L(1)=0$, $\ExpV L^4(1)<\infty$, and $Q_T$ be as in \eqref{def_fnctinl_Q_T}.
 Assume that $\hat{a}\in L_p(\mathbb{R})\cap L_2(\mathbb{R})$, $\hat{b}$ is of the form \eqref{def_Ftrfrm_b} with even function $b\in L_1(\mathbb{R})$ and $\hat{b}\in L_q(\mathbb{R})$ with
  \[
    1\le p,\,q\le2,\ \ \dfrac2{p}+\dfrac1{q}\ge\dfrac52,
  \]
 then
  \[
    T^{-\frac12}Q_T\ \Rightarrow\ N(0,\,\sigma^2),\ \ \text{as}\ \ T\rightarrow\infty,
  \]
 where $\sigma^2$ is given in \eqref{cov_mtrx_sigma^2}.
\end{lemma}

\begin{remark}
 It is important to note that conditions of Lemma \ref{lema_AN_lin_fnctnl} are given in frequency domain, while Lemma \ref{lema_AN_lin_fnctnl_Bai} employs the time domain conditions.
\end{remark}

Theorems similar to Lemmas \ref{lema_AN_lin_fnctnl} and \ref{lema_AN_lin_fnctnl_Bai} can be found in paper by Giraitis et al~\cite{GiTaTa_ANoQFoMD}, where the case of martingale-differences were considered. Overview of analogous results for different types of processes is given in the paper by Ginovyan et al~\cite{GiSaTa_tTP4TM&OaiIiP}.

Set
 \[
   \Delta_T(\varphi)=T^{-\frac12}\int\limits_{\mathbb{R}}\, \varepsilon_T(\lambda)\overline{s_T(\lambda,\,\widehat{\alpha}_T)} \varphi(\lambda)d\lambda.
 \]

\begin{lemma}\label{lema_cnv_Dlta_2_0}
 Suppose the conditions \textbf{A$_1$, A$_2$, C$_2$, N$_1$ -- N$_3$} are fulfilled, $\varphi(\lambda)$, $\lambda\in\mathbb{R}$, is a bounded differentiable function satisfying the relation \textbf{3)} of the condition \textbf{N$_4$(i)}, and moreover the derivative $\varphi'(\lambda)$, $\lambda\in\mathbb{R}$, is uniformly continuous on $\mathbb{R}$. Then
 \[
   \Delta_T(\varphi)\overset{\Prob}\longrightarrow 0\ \text{as}\ T\to\infty.
 \]
\end{lemma}

\begin{proof}
 Let $B_\sigma$ be the set of all bounded entire functions on $\mathbb{R}$ of exponential type $0\le\sigma<\infty$ (see Appendix \ref{app_Levitan_polnml}), and $\delta>0$ is an arbitrarily small number. Then there exists a function $\varphi_\sigma\in B_\sigma$, $\sigma=\sigma(\delta)$, such that
  \[
    \sup\limits_{\lambda\in\mathbb{R}}\, |\varphi(\lambda)-\varphi_\sigma(\lambda)|<\delta.
  \]
 Let $T_n(\varphi_\sigma;\,\lambda)=\sum\limits_{j=-n}^n\, c_j^{(n)}e^{\mathrm{i}j\frac\sigma{n}\lambda},\ n\geq 1$, be a sequence of the Levitan polynomials that corresponds to $\varphi_\sigma$. For any $\Lambda>0$ there exists $n_0=n_0(\delta,\,\Lambda)$ such that for $n>n_0$
  \[
    \sup\limits_{\lambda\in[-\Lambda,\Lambda]}\, |\varphi_\sigma-T_n(\varphi_\sigma;\,\lambda)|\leq\delta.
  \]
 Write
  \[
    \Delta_T(\varphi)= \Delta_T(\varphi-\varphi_\sigma) +\Delta_T(\varphi_\sigma-T_n)+\Delta_T(T_n),
  \]
  \[
   \begin{aligned}
     |\Delta_T(\varphi-\varphi_\sigma)|&\le \delta T^{-\frac12} \int\limits_{\mathbb{R}}\, \l|\varepsilon_T(\lambda) \overline{s_T(\lambda,\,\widehat{\alpha}_T)}\r|d\lambda\le\\
     &\le\delta T^{-\frac12}\l(\int\limits_{\mathbb{R}}\, \l|\varepsilon_T(\lambda)\r|^2d\lambda\r)^{\frac12} \l(\int\limits_{\mathbb{R}}\, \l| s_T(\lambda,\,\widehat{\alpha}_T)\r|^2d\lambda\r)^{\frac12}=\\
     &=2\pi\delta\l(\nu_T^*\r)^{\frac12}\Phi_T^{\frac12}(\widehat{\alpha}_T,\,\alpha_0) \le 2\pi c^{\frac12}_0\delta\l(\nu_T^*\r)^{\frac12}\l\|d_T(\alpha_0)\l(\widehat{\alpha}_T-\alpha_0\r)\r\|.
   \end{aligned}
  \]

 So, under the condition \textbf{C$_2$}, for any $\rho>0$
  \[
   \begin{aligned}
     \Prob&\l\{|\Delta_T(\varphi-\varphi_\sigma)|\ge\rho\r\}\le \\
      &\ \ \ \le \Prob\l\{\l\|d_T(\alpha_0)\l(\widehat{\alpha}_T-\alpha_0\r)\r\|\ge \dfrac{\rho}{2\pi c^{\frac12}_0\delta(B(0)+1)^{\frac12}}\r\}
     +\Prob\l\{\nu_T^*-B(0)>1\r\}=P_3+P_4.
   \end{aligned}
  \]

 The probability  $P_4\to0$, as $T\to\infty$, and the probability $P_3$ under the condition \textbf{N$_1$} for sufficiently large $T$ (we will write $T>T_0$) can be made less than a preassigned number by chosing $\delta>0$ for a fixed $\rho>0$.

  As far as the function $\varphi_\sigma\in B_\sigma$  and the corresponding sequence of Levitan polynomials $T_n$ are bounded by the same constant, we obtain
  \[
    |\Delta(\varphi_\sigma-T_n)|\le\delta T^{-\frac12} \int\limits_{-\Lambda}^{\Lambda}\, \l|\varepsilon_T(\lambda) \overline{s_T(\lambda,\,\widehat{\alpha}_T)}\r|d\lambda +2c(\varphi_\sigma)T^{-\frac12} \int\limits_{\mathbb{R}\backslash[-\Lambda,\Lambda]}\, \l|\varepsilon_T(\lambda) \overline{s_T(\lambda,\,\widehat{\alpha}_T)}\r|d\lambda=D_1+D_2.
  \]

 The integral in the term $D_1$ can be majorized by an integral over $\mathbb{R}$ and bounded as earlier. We have further
  \[
    \overline{s_T(\lambda,\,\widehat{\alpha}_T)}=(\mathrm{i}\lambda)^{-1} \l[e^{\mathrm{i}\lambda T}(g(T,\,\alpha_0)-g(T,\,\widehat{\alpha}_T)) -(g(0,\,\alpha_0)-g(0,\,\widehat{\alpha}_T)) -\overline{s_T'(\lambda,\,\widehat{\alpha}_T)}\r],
  \]
 where $\overline{s_T'(\lambda,\,\widehat{\alpha}_T)} =\int\limits_0^T\, e^{-\mathrm{i}\lambda t}(g'(t,\,\alpha_0)-g'(t,\,\widehat{\alpha}_T))dt$.

 Under the Lemma conditions
  \[
   \begin{aligned}
     T^{-\frac12}&\int\limits_{\mathbb{R}\backslash[-\Lambda,\Lambda]}\, |\varepsilon_T(\lambda)\overline{s_T(\lambda,\,\widehat{\alpha}_T)} |d\lambda\le T^{-\frac12}\l(\int\limits_{\mathbb{R}\backslash [-\Lambda,\Lambda]}\, \l|\varepsilon_T(\lambda)\r|^2 d\lambda\r)^{\frac12} \cdot\\ &\cdot\l(3\int\limits_{\mathbb{R}\backslash[-\Lambda,\Lambda]}\, \lambda^{-2} \l[\l|g(T,\,\alpha_0)-g(T,\,\widehat{\alpha}_T)\r|^2 +\l|g(0,\,\alpha_0)-g(0,\,\widehat{\alpha}_T)\r|^2 +\l|s_T'(\lambda,\,\widehat{\alpha}_T)\r|^2\r] d\lambda\r)^{\frac12}\le\\
     &\le \sqrt{3}\l(2\pi\nu_T^*\r)^{\frac12}\l(\sqrt{2}\Lambda^{-\frac12}\Bigl(\l|g(T,\,\widehat{\alpha}_T)- g(T,\,\alpha_0)\r| +\l|g(0,\,\widehat{\alpha}_T)-g(0,\,\alpha_0)\r|\Bigr)+\r.\\
     &\hspace*{81mm}\l.+\l(2\pi c'_0\r)^{\frac12}\Lambda^{-1}\l\|d_T(\alpha_0)\l(\widehat{\alpha}_T-\alpha_0\r)\r\|\r).
   \end{aligned}
  \]

 Obviously,
  \[
    g(T,\,\widehat{\alpha}_T)- g(T,\,\alpha_0)=\sum\limits_{i=1}^q\,g_i(T,\,\alpha^*_T), \l(\widehat{\alpha}_{iT}-\alpha_{i0}\r),
  \]
 $\alpha^*_T=\alpha_0+\eta\l(\widehat{\alpha}_T-\alpha_0\r)$, $\eta\in(0,\,1)$, $d_T(\alpha_0)\l(\alpha^*_T-\alpha_0\r)= \eta d_T(\alpha_0)\l(\widehat{\alpha}_T-\alpha_0\r)$, and for any $\rho>0$ and $i=\overline{1,q}$
  \[
   \begin{aligned}
    \Prob\Bigl\{\l|g_i(T,\,\alpha^*_T), \l(\widehat{\alpha}_{iT}-\alpha_{i0}\r)\r|\ge\rho\Bigr\} \le \Prob\Bigl\{\l|g_i(T,\,\alpha^*_T), \l(\widehat{\alpha}_{iT}-\alpha_{i0}\r)\r|\ge\rho,\ \l\|d_T(\alpha_0)\l(\widehat{\alpha}_T-\alpha_0\r)\r\|\le R\Bigr\}&+\\ +\Prob\Bigl\{\l\|d_T(\alpha_0)\l(\widehat{\alpha}_T-\alpha_0\r)\r\|> R\Bigr\}=&P_5+P_6.
   \end{aligned}
  \]

 By condition \textbf{N$_3$(i)} for any $R\ge0$
  \[
   \begin{aligned}
    P_5&\le \Prob\l\{\l(d_{iT}^{-1}(\alpha_0)\sup\limits_{t\in[0,T],\,\|u\|\le R}\, \l|g_i\l(t,\,\alpha_0+d_T^{-1}(\alpha_0)u\r)\r|\r)\cdot \l(d_{iT}^{-1}(\alpha_0)\l|\widehat{\alpha}_{iT}-\alpha_{i0}\r|\r)\ge\rho\r\}\le\\
    &\le\Prob\l\{T^{-\frac12}d_{iT}^{-1}(\alpha_0)\l|\widehat{\alpha}_{iT}-\alpha_{i0}\r|\ge\frac\rho{c^i(R)}\r\}\ \to\ 0,\ \ \text{as}\ \ T\to\infty,
   \end{aligned}
  \]
 according to \textbf{N$_1$} (or \textbf{C$_1$}). On the other hand, by condition \textbf{N$_1$} the value $R$ can be chosen so that for $T>T_0$ the probability $P_6$ becomes less that preassigned number.

 So,
  \[
    g(T,\,\widehat{\alpha}_T)- g(T,\,\alpha_0)\ \overset{\Prob}\longrightarrow\ 0,\ \ \text{as}\ \ T\to\infty,
  \]
 and, similarly, $g(0,\,\widehat{\alpha}_T)- g(0,\,\alpha_0)\ \overset{\Prob}\longrightarrow\ 0$, as $T\to\infty$.

 Moreover, for any $\rho>0$
  \[
   \begin{aligned}
     \Prob\l\{\Lambda^{-1}\l\|d_T(\alpha_0)\l(\widehat{\alpha}_T-\alpha_0\r)\r\|\ge\rho\r\}\le P_6+ \Prob\Bigl\{\Lambda^{-1}\l\|d_T(\alpha_0)\l(\widehat{\alpha}_T-\alpha_0\r)\r\|\ge\rho,\ \l\|d_T(\alpha_0)\l(\widehat{\alpha}_T-\alpha_0\r)\r\|\le R\Bigr\},
   \end{aligned}
  \]
 and the second probability is equal to zero, if $\Lambda>\frac R\rho$.

 Thus for any fixed $\rho>0$, similarly to the probability $P_3$, the probability $P_7=\Prob\{D_2\ge\rho\}$ for $T>T_0$ can be made less than preassigned number by the choice of the  value $\Lambda$.

 Consider
  \[
    \Delta_T(T_n)=T^{-\frac12}\sum\limits_{j=-n}^n\,c_j^{(n)} \int\limits_{\mathbb{R}}\,\varepsilon_T(\lambda) \overline{s_T(\lambda,\,\widehat{\alpha}_T)} e^{\mathrm{i}j\frac\sigma{n}\lambda}d\lambda,
  \]
  \[
    \overline{s_T(\lambda,\,\widehat{\alpha}_T)} e^{\mathrm{i}j\frac\sigma{n}\lambda}= \int\limits_{\frac{j\sigma}n}^{T+\frac{j\sigma}n}\, e^{\mathrm{i}\lambda t}\l(g\l(t-j\dfrac{\sigma}{n},\,\alpha_0\r) -g\l(t-j\dfrac{\sigma}{n},\,\widehat{\alpha}_T\r)\r)dt,\ j=\overline{-n,n}.
  \]
 It means that
  \[
   \begin{aligned}
     \Delta_T(T_n)=2\pi&\sum\limits_{j=1}^n\,c_j^{(n)}T^{-\frac12} \int\limits_{\frac{j\sigma}n}^T\,\varepsilon(t) \l(g\l(t-j\dfrac{\sigma}{n},\,\alpha_0\r) -g\l(t-j\dfrac{\sigma}{n},\,\widehat{\alpha}_T\r)\r)dt+\\
     &+2\pi\sum\limits_{j=-n}^0\,c_j^{(n)}T^{-\frac12} \int\limits_0^{T+\frac{j\sigma}n}\,\varepsilon(t) \l(g\l(t-j\dfrac{\sigma}{n},\,\alpha_0\r) -g\l(t-j\dfrac{\sigma}{n},\,\widehat{\alpha}_T\r)\r)dt.
   \end{aligned}
  \]

 For $j>0$ consider the value
  \[
   \begin{aligned}
     T^{-\frac12}&\int\limits_{\frac{j\sigma}n}^T\,\varepsilon(t) \l(g\l(t-j\dfrac{\sigma}{n},\,\widehat{\alpha}_T\r)- g\l(t-j\dfrac{\sigma}{n},\,\alpha_0\r)\r)dt=\\
     &=\sum\limits_{i=1}^q\,\l(T^{-\frac12}d_{iT}^{-1}(\alpha_0)\int\limits_{\frac{j\sigma}n}^T\, \varepsilon(t)g_i\l(t-j\dfrac{\sigma}{n},\,\alpha_0\r)dt\r) d_{iT}(\alpha_0)(\widehat{\alpha}_{iT} - \alpha_{i0})+\\
     &+\dfrac12\sum\limits_{i,k=1}^q\,\l(T^{-\frac12} \int\limits_{\frac{j\sigma}n}^T\,\varepsilon(t) g_{ik}\l(t-j\dfrac{\sigma}{n},\,\alpha_T^{*}\r)dt\r) (\widehat{\alpha}_{iT}-\alpha_{i0}) \l(\widehat{\alpha}_{kT}-\alpha_{k0}\r)=S_{1T}+\frac12S_{2T},
   \end{aligned}
  \]
 $\alpha_T^{*}=\alpha_0+\bar{\eta}\l(\widehat{\alpha}_T-\alpha_0\r)$, $\bar{\eta}\in(0,\,1)$.

 Note that for $i=\overline{1,q}$
  \[
    d_{iT}(\alpha_0)\l(\widehat{\alpha}_{iT}-\alpha_{i0}\r)\Rightarrow N(0,\,\Sigma_{_{LSE}}^{ii}),\ \ \text{as}\ \ T\to\infty,
  \]
 by the condition \textbf{N$_1$}. Moreover,
  \[
   \begin{aligned}
     \ExpV&\l(T^{-\frac12}d_{iT}^{-1}(\alpha_0)\int\limits_{\frac{j\sigma}n}^T\, \varepsilon(t)g_i\l(t-j\dfrac{\sigma}{n},\,\alpha_0\r)dt\r)^2=\\
     &=T^{-1}d_{iT}^{-2}(\alpha_0)\int\limits_{\frac{j\sigma}n}^T\int\limits_{\frac{j\sigma}n}^T\,B(t-s) g_i\l(t-j\dfrac{\sigma}{n},\,\alpha_0\r)g_i\l(s-j\dfrac{\sigma}{n},\,\alpha_0\r)dtds\le\\
     &\le \l(T^{-2}\int\limits_0^T\int\limits_0^T\,B^2(t-s)dtds\r)^{\frac12}=O\l(T^{-\frac12}\r),
   \end{aligned}
  \]
 since
  \[
    T^{-1}\int\limits_0^T\int\limits_0^T\,B^2(t-s)dtds\ \to\ 2\pi\|f\|_2^2,\ \ \text{as}\ \ T\to\infty.
  \]
 It means that the sum $S_{1T}\overset{\Prob}\longrightarrow0$, as $T\to\infty$.

 For the general term $S_{2T}^{ik}$ of the sum $S_{2T}$ and any $\rho>0$, $R>0$,
  \[
    \Prob\l\{\l|S_{2T}^{ik}\r|\ge\rho\r\}\le P_6+P_8,\ \ P_8=\Prob\Bigl\{\l|S_{2T}^{ik}\r|\ge\rho,\ \l\|d_T(\alpha_0)\l(\widehat{\alpha}_T-\alpha_0\r)\r\|\le R\Bigr\}.
  \]

 Under condition $\l\|d_T(\alpha_0)\l(\widehat{\alpha}_T-\alpha_0\r)\r\|\le R$ using assumptions \textbf{N$_3$(ii)} and \textbf{N$_3$(iii)} we get as in the estimation of the probability $P_5$
  \[
   \begin{aligned}
    \l|S_{2T}^{ik}\r|\le& \l(T^{-\frac12}\int\limits_{\frac{j\sigma}n}^T\, |\varepsilon(t)|dt\r)\cdot \l(d_{ik,T}^{-1}(\alpha_0)\sup\limits_{t\in[0,T],\,u\in v^c(R)}\, \l|g_{ik}\l(t,\,\alpha_0+d_T^{-1}(\alpha_0)u\r)\r|\r)\cdot\\
    &\cdot\Bigl(d_{iT}^{-1}(\alpha_0)d_{kT}^{-1}(\alpha_0) d_{ik,T}(\alpha_0)\Bigr)\cdot
    \l|d_{iT}(\alpha_0)(\widehat{\alpha}_{iT}-\alpha_{i0})\r|\cdot \l|d_{kT}(\alpha_0)(\widehat{\alpha}_{kT}-\alpha_{k0})\r|\le\\
    &\le c^{ik}(R)\tilde{c}^{ik}T^{-\frac32}\int\limits_0^T\, |\varepsilon(t)|dt\cdot\l|d_{iT}(\alpha_0)(\widehat{\alpha}_{iT}-\alpha_{i0})\r|\cdot \l|d_{kT}(\alpha_0)(\widehat{\alpha}_{kT}-\alpha_{k0})\r|.
   \end{aligned}
  \]
 By Lemma \ref{lema_int_eps^2}
  \[
    T^{-\frac32}\int\limits_0^T\, |\varepsilon(t)|dt\le\frac12T^{-\frac12}+\frac12 T^{-\frac32}\int\limits_0^T\, \varepsilon^2(t)dt\ \overset{\Prob}\longrightarrow\ 0,\ \ \text{as}\ \ T\to\infty.
  \]

  So, by condition \textbf{N$_1$} $P_8\to0$, as $T\to\infty$, that is $S_{2T}\ \overset{\Prob}\longrightarrow\ 0$, as $T\to\infty$. For $j\le0$ the reasoning is similar, and
  \[
    \Delta_T(T_n)\overset{\Prob}\longrightarrow0,\ T\to\infty.
  \]
\end{proof}

\begin{lemma}\label{lema_cnv_I_thta_T_2_thta_0}
 Let the function $\varphi(\lambda,\,\theta)w(\lambda)$ be continuous in $\theta\in\Theta^c$ for each fixed $\lambda\in\mathbb{R}$ with
  \[
    |\varphi(\lambda,\,\theta)|\le\varphi(\lambda),\ \theta\in\Theta^c,\ \text{and}\ \varphi(\cdot)w(\cdot)\in L_1(\mathbb{R}).
  \]
 If $\theta_T^{*}\overset{\Prob}\longrightarrow\theta_0$, then
  \[
    I\l(\theta_T^{*}\r)=\int\limits_{\mathbb{R}}\, \varphi\l(\lambda,\,\theta_T^{*}\r)w(\lambda)d\lambda\  \overset{\Prob}\longrightarrow\ \int\limits_{\mathbb{R}}\, \varphi(\lambda,\,\theta_0)w(\lambda)d\lambda=I(\theta_0).
  \]
\end{lemma}

\begin{proof}
 By a Lebesgue dominated convergence theorem the integral $I(\theta)$, $\theta\in\Theta^c$, is a continuous function. Further argument is standard. For any $\rho>0$ and $\varepsilon=\dfrac{\rho}2$ we find such a $\delta>0$, that $|I(\theta)-I(\theta_0)|<\varepsilon$ as $\|\theta-\theta_0\|<\delta$. Then
  \[
    \Prob\l\{|I(\theta_T^{*})-I(\theta_0)|\ge\rho\r\}=P_9+P_{10},
  \]
 where
  \[
    P_9=\Prob\l\{|I(\theta_T^{*})-I(\theta_0)|\ge\dfrac{\rho}2,\ \|\theta_T^{*}-\theta_0\|<\delta\r\}=0,
  \]
 due to the choice of $\varepsilon$, and
  \[
    P_{10}=\Prob\l\{|I(\theta_T^{*})-I(\theta_0)|\ge\dfrac{\rho}2,\ \|\theta_T^{*}-\theta_0\|\ge\delta\r\}\ \to\ 0,\ \ \text{as}\ \ T\to\infty.
  \]
\end{proof}

\begin{lemma}\label{lema_cnv_int_phi_eps_s_T}
 If the conditions \textbf{A$_1$, C$_2$} are satisfied and $\sup\limits_{\lambda\in \mathbb{R},\,\theta\in\Theta^c}\, |\varphi(\lambda,\,\theta)|=c(\varphi)<\infty$, then
  \[
   \begin{aligned}
     T^{-1}\int\limits_{\mathbb{R}}\,\varphi(\lambda,\,\theta_T^{*}) \varepsilon_T(\lambda)\overline{s_T(\lambda,\,\widehat{\alpha}_T)} d\lambda\ &\overset{\Prob}\longrightarrow\ 0,\ \ \text{as}\ \ T\to\infty,\\
     T^{-1}\int\limits_{\mathbb{R}}\,\varphi(\lambda,\,\theta_T^{*}) |s_T(\lambda,\,\widehat{\alpha}_T)|d\lambda\ &\overset{\Prob}\longrightarrow\  0,\ \ \text{as}\ \ T\to\infty.
   \end{aligned}
  \]
\end{lemma}

\begin{proof}
 These relations are similar to \eqref{J_T^1phi_2_0}, \eqref{J_T^2phi_2_0}, and can be obtained in the same way.
\end{proof}

\begin{lemma}\label{lema_cnv_int_phi_I_T^e_w}
 Let under conditions \textbf{A$_1$, A$_2$} there exists an even positive Lebesgue measurable function $v(\lambda)$, $\lambda\in\mathbb{R}$, and an even Lebesgue measurable in $\lambda$ for any fixed $\theta\in\Theta^c$ function $\varphi(\lambda,\,\theta)$, $(\lambda,\,\theta)\in\mathbb{R}\times\Theta^c$, such that\\
  \hspace*{15mm}\textbf{(i)} $\varphi(\lambda,\,\theta)v(\lambda)$ is uniformly continuous in  $(\lambda,\,\theta)\in\mathbb{R}\times\Theta^c$;\\
  \hspace*{14mm}\textbf{(ii)} $\underset{\lambda\in\mathbb{R}}\sup\,\dfrac{w(\lambda)}{v(\lambda)}<\infty$;\\
  \hspace*{12mm}\textbf{(iii)} $\underset{\lambda\in\mathbb{R},\ \theta\in\Theta^c}\sup\,|\varphi(\lambda,\,\theta)|w(\lambda)<\infty$.\\
 Suppose also that $\theta_T^{*}\overset{\Prob}\longrightarrow\theta_0$, then, as $T\to\infty$,
  \[
    \int\limits_{\mathbb{R}}\,I_T^{\varepsilon}(\lambda)\varphi(\lambda,\,\theta_T^{*})w(\lambda)d\lambda\ \overset{\Prob}\longrightarrow\ \int\limits_{\mathbb{R}}\,f(\lambda,\,\theta_0)\varphi(\lambda,\,\theta_0)w(\lambda)d\lambda.
  \]
\end{lemma}

\begin{proof}
 We have
  \[
   \begin{aligned}
     \int\limits_{\mathbb{R}}\, I_T^{\varepsilon}(\lambda)\varphi(\lambda,\,\theta_T^{*})w(\lambda)d\lambda =\int\limits_{\mathbb{R}}\,I_T^{\varepsilon}(\lambda) \bigl(\varphi(\lambda,\,\theta_T^{*}) -\varphi(\lambda,\,\theta_0)\bigr)v(\lambda) \dfrac{w(\lambda)}{v(\lambda)}d\lambda&+\\
     +\int\limits_{\mathbb{R}}\,I_T^{\varepsilon}(\lambda)\varphi(\lambda,\,\theta_0)w(\lambda)d\lambda&=I_5+I_6.
   \end{aligned}
  \]

 By Lemma \ref{lema_cnv_J_T^eps} and the condition \textbf{\textit{(iii)}}
  \begin{equation}
      I_6\ \overset{\Prob}\longrightarrow\ \int\limits_{\mathbb{R}}\, f(\lambda,\,\theta_0)\varphi(\lambda,\,\theta_0)w(\lambda) d\lambda,\ \ \text{as}\ \ T\to\infty.
   \label{cnv_I_4_2_int}
  \end{equation}
 On the other hand, for any $r>0$ under the condition \textbf{\textit{(i)}} there exists $\delta=\delta(r)$ such that for $\l\|\theta_T^{*}-\theta_0\r\|<\delta$
  \begin{equation}
     |I_5|\leq r \int\limits_{\mathbb{R}}\, I_T^{\varepsilon}\dfrac{w(\lambda)}{v(\lambda)}d\lambda,
   \label{ineq_I_3_by_r_int}
  \end{equation}
 and by the condition \textbf{\textit{(ii)}}
  \begin{equation}
      \int\limits_{\mathbb{R}}\,I_T^{\varepsilon} \dfrac{w(\lambda)}{v(\lambda)}d\lambda\ \overset{\Prob}\longrightarrow\  \int\limits_{\mathbb{R}}\,f(\lambda,\,\theta_0) \dfrac{w(\lambda)}{v(\lambda)}d\lambda.
   \label{cnv_int_I_T^e*w/v}
  \end{equation}

 The relations \eqref{cnv_I_4_2_int}--\eqref{cnv_int_I_T^e*w/v} prove the lemma.
\end{proof}

\begin{prfthm2}
 By definition of the MCE $\widehat{\theta}_T$, formally using the Taylor formula, we get
  \begin{equation}
       0=\nabla_\theta U_T(\widehat{\theta}_T,\,\widehat{\alpha}_T) =\nabla_\theta U_T(\theta_0,\,\widehat{\alpha}_T) +\nabla_\theta\nabla_\theta'U_T(\theta_T^{*},\,\widehat{\alpha}_T) (\widehat{\theta}_T-\theta_0).
   \label{nblaU_T_Tlr_exp}
  \end{equation}
  Since there is no vector Taylor formula, \eqref{nblaU_T_Tlr_exp} must be taken coordinatewise, that is each row of vector equality \eqref{nblaU_T_Tlr_exp} depends on its own random vector $\theta_T^{*}$, such that $\|\theta_T^{*}-\theta_0\| \leq\|\widehat{\theta}_T-\theta_0\|$. In turn, from \eqref{nblaU_T_Tlr_exp} we have formally
  \[
    T^{\frac12}(\widehat{\theta}_T-\theta_0)=\l(\nabla_\theta\nabla_\theta' U_T(\theta_T^{*},\,\widehat{\alpha}_T)\r)^{-1} \l(-T^{\frac12}\nabla_\theta U_T(\theta_0,\,\widehat{\alpha}_T)\r).
  \]

 As far as the condition \textbf{N$_4$} implies the possibility of differentiation under the sign of the integrals in \eqref{Wh_ctrst_fn}, then
  \begin{equation}
     \begin{aligned}
       -T^{\frac12}\nabla_\theta&U_T(\theta_0,\,\widehat{\alpha}_T)=-T^{\frac12}\int\limits_{\mathbb{R}}\, \l(\nabla_\theta \log f(\lambda,\,\theta_0)+\nabla_\theta\l(\dfrac1{f(\lambda,\,\theta_0)}\r)I_T(\lambda,\,\widehat{\alpha}_T)\r) w(\lambda)d\lambda=\\
       &=T^{\frac12}\int\limits_{\mathbb{R}}\, \l(\dfrac{\nabla_\theta f(\lambda,\,\theta_0)}{f^2(\lambda,\,\theta_0)} I_T^{\varepsilon}(\lambda)-\dfrac{\nabla_\theta f(\lambda,\,\theta_0)}{f(\lambda,\,\theta_0)}\r) w(\lambda)d\lambda+\\
       &+(2\pi)^{-1}T^{-\frac12}\int\limits_{\mathbb{R}}\, \l(2\re\l\{\varepsilon_T(\lambda) \overline{s_T(\lambda,\,\widehat{\alpha}_T)}\r\} +|s_T(\lambda,\,\widehat{\alpha}_T)|^2\r)\dfrac{\nabla_\theta f(\lambda,\,\theta_0)}{f^2(\lambda,\,\theta_0)} w(\lambda)d\lambda=\\
       &=A_T^{(1)}+ A_T^{(2)}+ A_T^{(3)}.
     \end{aligned}
   \label{nblaU_T_exp_by_A1-3}
  \end{equation}

 Similarly
  \begin{equation}
    \begin{aligned}
       \nabla_\theta\nabla_\theta'&U_T(\theta_T^{*},\,\widehat{\alpha}_T)=\int\limits_{\mathbb{R}}\, \l(\nabla_\theta \nabla_\theta' \log f(\lambda,\,\theta_T^{*})+\nabla_\theta\nabla_\theta' \l(\dfrac1{f(\lambda,\,\theta_T^{*})}\r)I_T(\lambda,\,\widehat{\alpha}_T)\r) w(\lambda)d\lambda=\\
        &=\int\limits_{\mathbb{R}}\,\l\{\l( \dfrac{\nabla_\theta \nabla_\theta'f(\lambda,\,\theta_T^{*})}{f(\lambda,\,\theta_T^{*})}  -\dfrac{\nabla_\theta f(\lambda,\,\theta_T^{*}) \nabla_\theta' f(\lambda,\,\theta_T^{*})}{f^2(\lambda,\,\theta_T^{*})}\r)\r.+\\
       &+\l(2\dfrac{\nabla_\theta f(\lambda,\,\theta_T^{*})\nabla_\theta' f(\lambda,\,\theta_T^{*})}{f^3(\lambda,\,\theta_T^{*})} -\dfrac{\nabla_\theta\nabla_\theta'f(\lambda,\,\theta_T^{*})} {f^2(\lambda,\,\theta_T^{*})}\r)\times\\
       &\ \ \  \times\l.(I_T^{\varepsilon}(\lambda)+(\pi T)^{-1}\re \{\varepsilon_T(\lambda) \overline{s_T(\lambda,\,\widehat{\alpha}_T)}\} +(2\pi T)^{-1} |s_T(\lambda,\,\widehat{\alpha}_T)|^2)\r\}w(\lambda)d\lambda=\\
       &\ \ \ \ \ \ \ =B_T^{(1)}+B_T^{(2)}+B_T^{(3)}+B_T^{(4)},
    \end{aligned}
   \label{nbla_nblaU_T_exp_by_B1-4}
  \end{equation}
 where the terms $B_T^{(3)}$ and $B_T^{(4)}$ contain values $\re\{\varepsilon_T(\lambda)\overline{s_T(\lambda,\widehat{\alpha}_T)}\}$ and $|s_T(\lambda,\widehat{\alpha}_T)|^2$, respectively.

 Bearing in mind the 1st part of the condition \textbf{N$_4$(i)}, we take in Lemma~\ref{lema_cnv_Dlta_2_0} the functions
  \[
      \varphi(\lambda)=\varphi_i(\lambda) =\dfrac{f_i(\lambda,\,\theta)}{f^2(\lambda,\,\theta)}w(\lambda),\ i=\overline{1,m}.
  \]
 Then in the formula \eqref{nblaU_T_exp_by_A1-3} $A_T^{(2)}\ \overset{\Prob}\longrightarrow\ 0$, as $T\to\infty$.

 Consider the term $A_T^{(3)}=(a_{iT}^{(3)})_{i=1}^m$, in the sum \eqref{nblaU_T_exp_by_A1-3}
  \[
    a_{iT}^{(3)}=(2\pi)^{-1}T^{-\frac12}\int\limits_{\mathbb{R}}\, |s_T(\lambda,\,\widehat{\alpha}_T)|^2\varphi_i(\lambda)d\lambda,
  \]
 where $\varphi_i(\lambda)$ are as before. Under conditions \textbf{C$_1$, C$_2$, N$_1$} and \textbf{1)} of \textbf{N$_4$(i)} $A_T^{(3)}\ \overset{\Prob}\longrightarrow\ 0$, as $T\rightarrow\infty$, because
  \[
    |a_{iT}^{(3)}|\le c(\varphi_i)T^{-\frac12} \Phi_T(\widehat{\alpha}_T,\,\alpha_0)\le c(\varphi_i)c_0 \|T^{-\frac12}d_T(\alpha_0)\l(\widehat{\alpha}_T-\alpha_0\r)\|\;\|d_T(\alpha_0)\l(\widehat{\alpha}_T- \alpha_0\r)\|\ \overset{\Prob}\longrightarrow\ 0,\ \ \text{as}\ \ T\rightarrow\infty.
  \]

 Examine the behaviour of the terms $B_T^{(1)}-B_T^{(4)}$ in formula \eqref{nbla_nblaU_T_exp_by_B1-4}. Under conditions \textbf{C$_1$} and \textbf{N$_4$(iii)} we can use Lemma \ref{lema_cnv_I_thta_T_2_thta_0} with functions
  \[
    \varphi(\lambda,\,\theta)=\varphi_{ij}(\lambda,\,\theta)=\dfrac{f_{ij}(\lambda,\,\theta)} {f(\lambda,\,\theta)},\ \dfrac{f_i(\lambda,\,\theta) f_j(\lambda,\,\theta)}{f^2(\lambda,\,\theta)},\ i,j=\overline{1,m},
  \]
 to obtain the convergence
  \begin{equation}
      B_T^{(1)}\ \overset{\Prob}\longrightarrow\ \int\limits_{\mathbb{R}}\, \l(\dfrac{\nabla_\theta\nabla_\theta'f(\lambda,\,\theta_0)} {f(\lambda,\,\theta_0)}-\dfrac{\nabla_\theta f(\lambda,\,\theta_0) \nabla_\theta'f(\lambda,\,\theta_0)}{f^2(\lambda,\,\theta_0)}\r) w(\lambda)d\lambda,\ \text{as}\ T\rightarrow\infty.
   \label{cnv_B_T^1_2_int}
  \end{equation}

 Under the condition \textbf{N$_5$(i)}  we can use Lemma \ref{lema_cnv_int_phi_eps_s_T} with functions
  \[
    \varphi(\lambda,\,\theta)=\varphi_{ij}(\lambda,\,\theta)= \dfrac{f_{ij}(\lambda,\,\theta)}{f^2(\lambda,\,\theta)} w(\lambda),\ \dfrac{f_i(\lambda,\,\theta) f_j(\lambda,\,\theta)}{f^3(\lambda,\,\theta)},\ i,j=\overline{1,m},
  \]
 to obtain that
  \[
    B_T^{(3)}\ \overset{\Prob}\longrightarrow\ 0,\ \ B_T^{(4)}\ \overset{\Prob}\longrightarrow\ 0,\ \ \text{as}\ \ T\to\infty.
  \]

 Under conditions \textbf{C$_1$} and \textbf{N$_5$}
  \begin{equation}
       B_T^{(2)}\ \overset{\Prob}\longrightarrow\ \int\limits_{\mathbb{R}}\, \l(2\dfrac{\nabla_\theta f(\lambda,\,\theta_0)\nabla_\theta' f(\lambda,\,\theta_0)}{f^2(\lambda,\,\theta_0)} -\dfrac{\nabla_\theta\nabla_\theta'f(\lambda,\,\theta_0)} {f(\lambda,\,\theta_0)}\r)w(\lambda)d\lambda,
   \label{cnv_B_T^2_2_int}
  \end{equation}
 if we take in Lemma \ref{lema_cnv_int_phi_I_T^e_w} in conditions \textbf{\textit{(i)}} and \textbf{\textit{(iii)}}
  \[
    \varphi(\lambda,\,\theta)=\varphi_{ij}(\lambda,\,\theta)= \dfrac{f_i(\lambda,\,\theta)f_j(\lambda,\,\theta)} {f^3(\lambda,\,\theta)},\ \dfrac{f_{ij}(\lambda,\,\theta)} {f^2(\lambda,\,\theta)}\ i,j=\overline{1,m}.
  \]

 So, under conditions \textbf{C$_1$, C$_2$, N$_4$(iii)} and \textbf{N$_5$}
  \begin{equation}
    \begin{aligned}
       \nabla_\theta\nabla_\theta'U_T(\theta_T^{*},\,\widehat{\alpha}_T)\ \overset{\Prob}\longrightarrow\ &\int\limits_{\mathbb{R}}\,\dfrac{\nabla_\theta f(\lambda,\,\theta_0)\nabla_\theta' f(\lambda,\,\theta_0)}{f^2(\lambda,\,\theta_0)}w(\lambda)d\lambda=\\
       =&\int\limits_{\mathbb{R}}\,\nabla_\theta\log f(\lambda,\,\theta_0)\nabla_\theta'\log f(\lambda,\,\theta_0)w(\lambda)d\lambda=W_1(\theta_0),
    \end{aligned}
   \label{cnv_nbla_nblaU_T_2_W1}
  \end{equation}
 because $W_1(\theta_0)$ is the sum of the right hand sides of \eqref{cnv_B_T^1_2_int} and \eqref{cnv_B_T^2_2_int}.

 From the facts obtained, it follows that for the proof of Theorem \ref{thm_MCE_asym_norm} it is necessary to study an asymptotic behaviour of vector $A_T^{(1)}$ from \eqref{nblaU_T_exp_by_A1-3}:
  \[
    A_T^{(1)}= T^{\frac12}\int\limits_{\mathbb{R}}\,\l(\dfrac{\nabla_\theta f(\lambda,\,\theta_0)}{f^2(\lambda,\,\theta_0)}I_T^{\varepsilon}(\lambda)-\dfrac{\nabla_\theta f(\lambda,\,\theta_0)}{f(\lambda,\,\theta_0)}\r) w(\lambda)d\lambda.
  \]
 We will take
  \[
   \begin{aligned}
    \varphi_i(\lambda)=&\dfrac{f_i(\lambda,\,\theta_0)}{f^2(\lambda,\,\theta_0)}w(\lambda),\ i=\overline{1,m},\\
    \Psi(\lambda)=&\sum\limits_{i=1}^m\,u_i\varphi_i(\lambda),\ \mathrm{u}=\l(u_1,\,\ldots,\,u_m\r)\in\mathbb{R}^m,\\
    Y_T=&\int\limits_{\mathbb{R}}\,I_T^{\varepsilon}(\lambda)\Psi(\lambda)d\lambda,\ \ \ Y=\int\limits_{\mathbb{R}}\,f(\lambda,\,\theta_0)\Psi(\lambda)d\lambda,
   \end{aligned}
  \]
 and write
  \[
    \l<A_T^{(1)},\,\mathrm{u}\r>=T^{\frac12}(Y_T-\ExpV Y_T)+T^{\frac12}(\ExpV Y_T-Y).
  \]

 Under conditions \textbf{1)} and \textbf{2)} of \textbf{N$_4$(i)} \cite{Ben_oEoEoSFoSP,Ibr_aEoSGPSF} for any $u\in\mathbb{R}^m$
  \begin{equation}
      T^{\frac12}(\ExpV Y_T-Y)\ \longrightarrow\ 0, \ \ \text{as}\ \ T\to\infty.
   \label{cnv_T12_EYT-T_2_0}
  \end{equation}

 On the other hand
  \[
    T^{\frac12}(Y_T-\ExpV Y_T)=T^{-\frac12}\int\limits_0^T\int\limits_0^T\, \l(\varepsilon(t)\varepsilon(s) -B(t-s)\r)\hat{b}(t-s)dtds
  \]
 with
  \[
    \hat{b}(t)=\int\limits_{\mathbb{R}}\,e^{\mathrm{i}\lambda t}\,(2\pi)^{-1}\Psi(\lambda)d\lambda .
  \]
 Thus we can apply Lemma \ref{lema_AN_lin_fnctnl} taking $b(\lambda)=(2\pi)^{-1}\Psi(\lambda)$ in the formula \eqref{cov_mtrx_sigma^2} to obtain for any $u\in\mathbb{R}^m$
  \begin{equation}
     T^{\frac12}(Y_T-\ExpV Y_T)\ \Rightarrow\ N(0,\,\sigma^2),\ \ \text{as}\ \ T\rightarrow\infty,
   \label{asymNorm_T12_YT-EYT}
  \end{equation}
 where
  \[
   \begin{aligned}
     \sigma^2=&4\pi\int\limits_{\mathbb{R}}\,\Psi^2(\lambda) f^2(\lambda,\,\theta_0)d\lambda +\gamma_2\l(\int\limits_{\mathbb{R}}\,\Psi(\lambda) f(\lambda,\,\theta_0)d\lambda\r)^2.
   \end{aligned}
  \]

 The relations \eqref{cnv_T12_EYT-T_2_0} and \eqref{asymNorm_T12_YT-EYT} are equivalent to the convergence\
  \begin{equation}
     A_T^{(1)}\ \Rightarrow\ N\l(0,\,W_2(\theta_0)+V(\theta_0)\r),\ \ \text{as}\ \ T\rightarrow\infty.
   \label{A_T^1_2_N0,W2+V}
  \end{equation}

 From \eqref{cnv_nbla_nblaU_T_2_W1} and \eqref{A_T^1_2_N0,W2+V} it follows \eqref{cov_mtrx_W}.
\end{prfthm2}

\begin{remark}
 From the conditions of Theorem \ref{thm_MCE_asym_norm} it follows also the fulfillment of Lemma \ref{lema_AN_lin_fnctnl_Bai} conditions for functions $\hat{a}$ and $\hat{b}$. Really by condition \textbf{A$_1$} $\hat{a}\in L_1(\mathbb{R})\cap L_2(\mathbb{R})$ and we can take $p=1$ in Lemma \ref{lema_AN_lin_fnctnl_Bai}. On the other hand, if we look at $b=(2\pi)^{-1}\Psi$ as at an original of the Fourier transform, from \textbf{N$_4$(i)1)} we have $b\in L_1(\mathbb{R})\cap L_2(\mathbb{R})$. Then according to the Plancherel theorem $\hat{b}\in L_2(\mathbb{R})$ and we can take $q=2$ in Lemma \ref{lema_AN_lin_fnctnl_Bai}. Thus
  \[
    \frac2p+\frac1q=\frac52,
  \]
 and conclusion of Lemma \ref{lema_AN_lin_fnctnl_Bai} is true.
\end{remark}
%-----------------------------------------------------------
\section{Example. The motion of a pendulum in a turbulent fluid}\label{sec_Example}

$\indent$First of all we review a number of results discussed in Parzen~\cite{Parz_SP}, Anh et al.~\cite{AnhHeLeo_DMoLMPDbLP}, Leonenko and Papić~\cite{LeoPap_CPoCTAPDnIoSS}, see also references therein.

We examine the stationary Lévy-driven continuous-time autoregressive process $\varepsilon(t),\ t\in\mathbb{R}$, of the order two ( $CAR(2)$-process ) in the under-damped case (see \cite{LeoPap_CPoCTAPDnIoSS} for details).

The motion of a pendulum is described by the equation
 \begin{equation}
    \ddot{\varepsilon}(t)+2\alpha\dot{\varepsilon}(t)+\l(\omega^2+\alpha^2\r)\varepsilon(t)=\dot{L}(t),\ t\in\mathbb{R},
  \label{PendTurbFl_dif_eq}
 \end{equation}
in which $\varepsilon(t)$ is the replacement from its rest position, $\alpha$ is a damping factor, $\dfrac{2\pi}\omega$ is the damped period of the pendulum (see, i.e., \cite{Parz_SP}, p.~111-113).

We consider the Green function solution of the equation \eqref{PendTurbFl_dif_eq}, in which $\dot{L}$ is the Lévy noise, i.e.  the derivative of a Lévy process in the distribution sense (see \cite{AnhHeLeo_DMoLMPDbLP} and \cite{LeoPap_CPoCTAPDnIoSS} for details). The solution can be defined as the linear process
 \[
   \varepsilon(t)=\int\limits_{\mathbb{R}}\,\hat{a}(t-s)dL(s),\ t\in\mathbb{R},
 \]
where the Green function
 \begin{equation}
   \hat{a}(t)=e^{-\alpha t}\, \frac{\sin(\omega t)}\omega\,\mathbb{I}_{[0,\,\infty)}(t),\ \alpha>0.
  \label{hat_a_repr_exle}
 \end{equation}

Assuming $\ExpV L(1)=0$, $d_2=\ExpV L^2(1)<\infty$, we obtain
 \begin{equation}
   B(t)=d_2\int\limits_0^\infty\,\hat{a}(t+s)\hat{a}(s)ds=\frac{d_2}{4(\alpha^2+\omega^2)}\,e^{-\alpha|t|}\, \l(\frac{\sin(\omega|t|)}\omega+\frac{\cos(\omega t)}\alpha\r).
  \label{CovFn_eps4exle}
 \end{equation}
The formula \eqref{CovFn_eps4exle} for the covariance function of the process $\varepsilon$ corresponds to the formula (2.12) in \cite{LeoPap_CPoCTAPDnIoSS} for the correlation function
 \[
   \Corr\l(\varepsilon(t),\,\varepsilon(0)\r)=\frac{B(t)}{B(0)}= e^{-\alpha|t|}\,\l(\cos(\omega t)+\frac\alpha\omega\sin(\omega|t|)\r).
 \]

On the other hand for $\hat{a}(t)$  given by \eqref{hat_a_repr_exle}
 \[
   a(\lambda)=\int\limits_0^\infty\,e^{-i\lambda t}\hat{a}(t)dt=\frac1{\alpha^2+\omega^2-\lambda^2+2i\alpha\lambda}.
 \]
Then the positive spectral density of the stationary process $\varepsilon$ can be written as (compare with \cite{Parz_SP})
 \begin{equation}
   f_2(\lambda)=\frac{d_2}{2\pi}\l|a(\lambda)\r|^2=\frac{d_2}{2\pi}\cdot\frac1{\l(\lambda^2-\alpha^2-\omega^2\r)^2+4\alpha^2\lambda^2},\ \lambda\in\mathbb{R}.
  \label{spec_dens_eps_exle}
 \end{equation}

It is convenient to rewrite \eqref{spec_dens_eps_exle} in the form
 \begin{equation}
   f_2(\lambda)=f(\lambda,\,\theta)=\frac1{2\pi}\cdot\frac{\beta}{\l(\lambda^2-\alpha^2-\gamma^2\r)^2+4\alpha^2\lambda^2},\ \lambda\in\mathbb{R},
  \label{spec_dens_eps_repr_exle}
 \end{equation}
where $\alpha=\theta_1$ is a damping factor, $\beta=-\varkappa^{(2)}(0)=d_2(\theta_2)=\theta_2$, $\gamma=\omega=\theta_3$ is a damped cyclic frequency of the pendulum oscillations. Suppose that
 \[
   \theta=\l(\theta_1,\,\theta_2,\,\theta_3\r)=\l(\alpha,\,\beta,\,\gamma\r)\in\Theta=\l(\underline{\alpha},\,\overline{\alpha}\r)\times \l(\underline{\beta},\,\overline{\beta}\r)\times \l(\underline{\gamma},\,\overline{\gamma}\r),\ \underline{\alpha},\underline{\beta},\underline{\gamma}>0,\ \overline{\alpha},\overline{\beta},\overline{\gamma}<\infty.
 \]

The condition \textbf{C$_3$} is fulfilled for spectral density \eqref{spec_dens_eps_repr_exle}.

Assume that
 \[
   w(\lambda)=\l(1+\lambda^2\r)^{-a},\ \lambda\in\mathbb{R},\ a>0.
 \]
More precisely the value of $a$ will be chosen below.

Obviously the functions $w(\lambda)\log f(\lambda,\,\theta)$, $\frac{w(\lambda)}{f(\lambda,\,\theta)}$ are continuous on $\mathbb{R}\times\Theta^c$.  For any $\Lambda>0$ the function $\l|\log f(\lambda,\,\theta)\r|$ is bounded on the set $[-\Lambda,\,\Lambda]\times\Theta^c$. The number $\Lambda$ can be chosen so that for $\mathbb{R}\backslash[-\Lambda,\,\Lambda]$
 \[
   1<\frac{8\pi}{\overline{\beta}}\underline{\alpha}^2\lambda^2\le f^{-1}(\lambda,\,\theta) \le\frac{2\pi}{\underline{\beta}}\l(2\l(\lambda^4+\l(\overline{\alpha}^2+\overline{\gamma}^2\r)^2\r)+4\overline{\alpha}^2\lambda^2\r).
 \]
Thus the function $Z_1(\lambda)$ in the condition \textbf{C$_4$(i)} exists.

As for condition \textbf{C$_4$(ii)}, if $a\ge2$, then
 \[
   \sup\limits_{\lambda\in\mathbb{R},\,\theta\in\Theta^c}\,\frac{w(\lambda)}{f(\lambda,\,\theta)}<\infty.
 \]

As a function $v$ in condition \textbf{C$_5$} we take
 \[
   v(\lambda)=\l(1+\lambda^2\r)^{-b},\ \lambda\in\mathbb{R},\ b>0.
 \]

Obviously, if $a\ge b$, then $\sup\limits_{\lambda\in\mathbb{R}}\,\frac{w(\lambda)}{v(\lambda)}<\infty$ (condition \textbf{C$_5$(ii)}), and the function $\frac{v(\lambda)}{f(\lambda,\,\theta)}$ is uniformly continuous in $(\lambda,\,\theta)\in\mathbb{R}\times\Theta^c$, if $b\ge2$ (condition \textbf{C$_5$(i)}).

Further it will be helpful to use the notation $s(\lambda)=\l(\lambda^2-\alpha^2-\gamma^2\r)^2+4\alpha^2\lambda^2$. Then
 \begin{equation}
   \begin{aligned}
     f_\alpha(\lambda,\,\theta)&=\frac\partial{\partial\alpha}\,f(\lambda,\,\theta)=-\frac{2\alpha\beta}{\pi}\l(\lambda^2+\alpha^2+\gamma^2\r)s^{-2}(\lambda);\\
     f_\beta(\lambda,\,\theta)&=\frac\partial{\partial\beta}\,f(\lambda,\,\theta)=\l(2\pi s(\lambda)\r)^{-2};\\
     f_\gamma(\lambda,\,\theta)&=\frac\partial{\partial\gamma}\,f(\lambda,\,\theta)=\frac{2\beta\gamma}{\pi}\l(\lambda^2-\alpha^2-\gamma^2\r)s^{-2}(\lambda).
   \end{aligned}
  \label{drvtv_spec_dens4exle}
 \end{equation}

To check the condition \textbf{N$_4$(i)1)} consider the functions
 \begin{equation}
   \begin{aligned}
     \varphi_\alpha(\lambda)&=\frac{f_\alpha(\lambda,\,\theta)}{f^2(\lambda,\,\theta)}w(\lambda) =-\frac{4\pi\alpha}\beta\l(\lambda^2+\alpha^2+\gamma^2\r)w(\lambda);\\
     \varphi_\beta(\lambda)&=\frac{f_\beta(\lambda,\,\theta)}{f^2(\lambda,\,\theta)}w(\lambda)=\frac{2\pi}{\beta^2}s(\lambda)w(\lambda);\\
     \varphi_\gamma(\lambda)&=\frac{f_\gamma(\lambda,\,\theta)}{f^2(\lambda,\,\theta)}w(\lambda) =\frac{8\pi\gamma}\beta\l(\lambda^2-\alpha^2-\gamma^2\r)w(\lambda).
   \end{aligned}
  \label{def_phis4exle}
 \end{equation}
Then the condition \textbf{N$_4$(i)1)} is satisfied for $\varphi_\alpha$ and $\varphi_\gamma$ when $a>\frac32$, for $\varphi_\beta$ when $a>\frac52$. The same values of $a$ are sufficient also to meet the condition \textbf{N$_4$(i)2)}.

To verify \textbf{N$_4$(i)3)} fix $\theta\in\Theta^c$ and denote by $\varphi(\lambda)$, $\lambda\in\mathbb{R}$, any of the continuous functions $\varphi_\alpha(\lambda)$, $\varphi_\beta(\lambda)$, $\varphi_\gamma(\lambda)$, $\lambda\in\mathbb{R}$. Suppose $|1-\eta|<\delta<\frac12$. Then
 \[
    \sup\limits_{\lambda\in\mathbb{R}}\,\l|\varphi(\eta\lambda)-\varphi(\lambda)\r| =\max\l(\sup\limits_{\eta|\lambda|\le\Lambda}\,\l|\varphi(\eta\lambda)-\varphi(\lambda)\r|,\ \sup\limits_{\eta|\lambda|>\Lambda}\,\l|\varphi(\eta\lambda)-\varphi(\lambda)\r|\r)=\max\l(s_1,\,s_2\r),
 \]
 \[
   s_2\le\sup\limits_{|\lambda|>\Lambda}\,\l|\varphi(\lambda)\r|+\sup\limits_{\eta|\lambda|>\Lambda}\,\l|\varphi(\lambda)\r|= s_3+s_4.
 \]
By the properties of the functions $\varphi$ under assumption $a>\frac52$ for any $\varepsilon>0$ there exists $\Lambda=\Lambda(\varepsilon)>0$ such that for $|\lambda|>\frac23\Lambda$\ \ $|\varphi(\lambda)|<\frac\varepsilon2$. So, $s_3\le\frac\varepsilon2$. We have also $s_4\le\underset{|\lambda|>\frac23\Lambda}\sup\,|\varphi(\lambda)|\le\frac\varepsilon2$. On the other hand,
 \[
   s_1\le \sup\limits_{|\lambda|<2\Lambda}\,\l|\varphi(\eta\lambda)-\varphi(\lambda)\r|,\ \ |\eta\lambda-\lambda|\le2\Lambda\delta=\delta',
 \]
and by the proper choice of $\delta$
 \[
   s_1\le\sup\limits_{\substack{\lambda_1,\lambda_2\in[-2\Lambda,\,2\Lambda]\\ \l|\lambda_1-\lambda_2\r|<\delta'}}\, \l|\varphi(\lambda_1)-\varphi(\lambda_2)\r|<\varepsilon,
 \]
and condition \textbf{N$_4$(i)3)} is met.

Using \eqref{def_phis4exle} we get for any $\theta\in\Theta^c$, as $\lambda\to\infty$,
 \[
    \begin{aligned}
     \varphi_\alpha'(\lambda)&=-\frac{8\pi\alpha}\beta\,\lambda w(\lambda)-\frac{4\pi\alpha}\beta\l(\lambda^2+\alpha^2+\gamma^2\r)w'(\lambda) =O\l(\lambda^{-2a+1}\r);\\
     \varphi_\beta'(\lambda)&=\frac{2\pi}{\beta^2}\bigl(s'(\lambda)w(\lambda)+s(\lambda)w'(\lambda)\bigr)=O\l(\lambda^{-2a+3}\r);\\
     \varphi_\gamma'(\lambda)&=\frac{16\pi\gamma}\beta\,\lambda w(\lambda)+\frac{8\pi\gamma}\beta\l(\lambda^2-\alpha^2-\gamma^2\r)w'(\lambda) =O\l(\lambda^{-2a+1}\r).
   \end{aligned}
 \]
Therefore for $a>\frac32$ these derivatives are uniformly continuous on $\mathbb{R}$ (condition \textbf{N$_4$(i)4)}). So, to satisfy condition \textbf{N$_4$(i)} we can take weight function $w(\lambda)$ with $a>\frac52$.

The check of assumption \textbf{N$_4$(ii)} is similar to the check of \textbf{C$_4$(i)}.

As $\lambda\to\infty$, uniformly in $\theta\in\Theta^c$
 \begin{equation}
   \begin{aligned}
     \frac{\l|f_\alpha(\lambda,\,\theta)\r|}{f(\lambda,\,\theta)}w(\lambda)&=\l|\varphi_\alpha(\lambda)\r|f(\lambda,\,\theta)w(\lambda) =2\alpha\l(\lambda^2+\alpha^2+\gamma^2\r)s^{-1}(\lambda)w(\lambda)=O\l(\lambda^{-2a-2}\r);\\
     \frac{\l|f_\beta(\lambda,\,\theta)\r|}{f(\lambda,\,\theta)}w(\lambda)&=\varphi_\beta(\lambda)f(\lambda,\,\theta)w(\lambda) =\beta^{-1}w(\lambda)=O\l(\lambda^{-2a}\r);\\
     \frac{\l|f_\gamma(\lambda,\,\theta)\r|}{f(\lambda,\,\theta)}w(\lambda)&=\l|\varphi_\gamma(\lambda)\r|f(\lambda,\,\theta)w(\lambda) =4\gamma\l|\lambda^2-\alpha^2-\gamma^2\r|s^{-1}(\lambda)w(\lambda)=O\l(\lambda^{-2a-2}\r).
   \end{aligned}
  \label{ordr_phis_f_w4exle}
 \end{equation}
On the other hand, for any $\Lambda>0$ the functions \eqref{ordr_phis_f_w4exle} are bounded on the sets $[-\Lambda,\,\Lambda]\times\Theta^c$.

To check \textbf{N$_4$(iii)} note first of all that the functions uniformly in $\theta\in\Theta^c$, as $\lambda\to\infty$,
 \begin{equation}
   \begin{aligned}
     \frac{f_\alpha^2(\lambda,\,\theta)}{f^2(\lambda,\,\theta)}w(\lambda)&=\varphi_\alpha(\lambda)f(\lambda,\,\theta) =8\alpha^2\l(\lambda^2+\alpha^2+\gamma^2\r)^2s^{-2}(\lambda)w(\lambda)=O\l(\lambda^{-2a-4}\r);\\
     \frac{f_\beta^2(\lambda,\,\theta)}{f^2(\lambda,\,\theta)}w(\lambda)&=\varphi_\beta(\lambda)f(\lambda,\,\theta) =\beta^{-2}w(\lambda)=O\l(\lambda^{-2a}\r);\\
     \frac{f_\gamma^2(\lambda,\,\theta)}{f^2(\lambda,\,\theta)}w(\lambda)&=\varphi_\gamma(\lambda)f(\lambda,\,\theta) =16\gamma^2\l(\lambda^2-\alpha^2-\gamma^2\r)^2s^{-2}(\lambda)w(\lambda)=O\l(\lambda^{-2a-4}\r).
   \end{aligned}
  \label{ordr_phis^2/f^2w4exle}
 \end{equation}
These functions are continuous on $\mathbb{R}\times\Theta^c$, as well as the functions
 \begin{equation}
   \begin{aligned}
     \frac{f_\alpha(\lambda,\,\theta)f_\beta(\lambda,\,\theta)}{f^2(\lambda,\,\theta)}w(\lambda)&=\varphi_\alpha(\lambda)f_\beta(\lambda,\,\theta) =-\frac{2\alpha}\beta\l(\lambda^2+\alpha^2+\gamma^2\r)s^{-1}(\lambda)w(\lambda);\\
     \frac{f_\alpha(\lambda,\,\theta)f_\gamma(\lambda,\,\theta)}{f^2(\lambda,\,\theta)}w(\lambda)&=\varphi_\alpha(\lambda)f_\gamma(\lambda,\,\theta) =-8\alpha\gamma \l(\lambda^4-\l(\alpha^2+\gamma^2\r)^2\r)s^{-2}(\lambda)w(\lambda);\\
     \frac{f_\beta(\lambda,\,\theta)f_\gamma(\lambda,\,\theta)}{f^2(\lambda,\,\theta)}w(\lambda)&=\varphi_\beta(\lambda)f_\gamma(\lambda,\,\theta) =\frac{4\gamma}\beta\l(\lambda^2-\alpha^2-\gamma^2\r)s^{-1}(\lambda)w(\lambda).
   \end{aligned}
  \label{ordr_phi*phi/f^2w4exle}
 \end{equation}
Moreover, uniformly in $\theta\in\Theta^c$, as $\lambda\to\infty$,
 \begin{equation}
   \begin{aligned}
     \frac{f_{\alpha\alpha}(\lambda,\,\theta)}{f(\lambda,\,\theta)}w(\lambda)&=-4\l(\lambda^2+3\alpha^2+\gamma^2\r)s^{-1}(\lambda)w(\lambda) +8\alpha\l(\lambda^2+\alpha^2+\gamma^2\r)s^{-2}(\lambda)s_\alpha'(\lambda)w(\lambda) =O\l(\lambda^{-2a-2}\r);\\
     \frac{f_{\beta\beta}(\lambda,\,\theta)}{f(\lambda,\,\theta)}w(\lambda)&=0;\\
     \frac{f_{\gamma\gamma}(\lambda,\,\theta)}{f(\lambda,\,\theta)}w(\lambda)&=4\l(\lambda^2-\alpha^2-3\gamma^2\r)s^{-1}(\lambda)w(\lambda) -8\gamma\l(\lambda^2-\alpha^2-\gamma^2\r)s^{-2}(\lambda)s_\gamma'(\lambda)w(\lambda) =O\l(\lambda^{-2a-2}\r);\\
     \frac{f_{\alpha\beta}(\lambda,\,\theta)}{f(\lambda,\,\theta)}w(\lambda)&=-\frac{4\alpha}\beta\l(\lambda^2+\alpha^2+\gamma^2\r)s^{-1}(\lambda)w(\lambda) =O\l(\lambda^{-2a-2}\r);\\
     \frac{f_{\alpha\gamma}(\lambda,\,\theta)}{f(\lambda,\,\theta)}w(\lambda)&=-8\alpha\gamma s^{-1}(\lambda)w(\lambda)+ 16\alpha\gamma\l(\lambda^4-\l(\alpha^2+\gamma^2\r)^2\r)s^{-2}(\lambda)w(\lambda) =O\l(\lambda^{-2a-4}\r);\\
     \frac{f_{\beta\gamma}(\lambda,\,\theta)}{f(\lambda,\,\theta)}w(\lambda)&=\frac{4\gamma}\beta\l(\lambda^2-\alpha^2-\gamma^2\r)s^{-1}(\lambda)w(\lambda) =O\l(\lambda^{-2a-2}\r).
   \end{aligned}
  \label{ordr_2nd_der_f/f*w4exle}
 \end{equation}

Note that the functions \eqref{ordr_2nd_der_f/f*w4exle} are continuous on $\mathbb{R}\times\Theta^c$ as well as functions \eqref{ordr_phis^2/f^2w4exle} and \eqref{ordr_phi*phi/f^2w4exle}. Therefore the condition  \textbf{N$_4$(iii)} is fulfilled.

Let us verify the condition \textbf{N$_5$(i)}. According to equation \eqref{ordr_phis^2/f^2w4exle}, uniformly in $\theta\in\Theta^c$, as $\lambda\to\infty$,
 \begin{equation}
   \begin{aligned}
     \frac{f_\alpha^2(\lambda,\,\theta)}{f^3(\lambda,\,\theta)}w(\lambda)& =\frac{16\pi\alpha^2}\beta\l(\lambda^2+\alpha^2+\gamma^2\r)^2s^{-1}(\lambda)w(\lambda)=O\l(\lambda^{-2a}\r);\\
     \frac{f_\beta^2(\lambda,\,\theta)}{f^3(\lambda,\,\theta)}w(\lambda)& =\frac{2\pi}{\beta^3}s(\lambda)w(\lambda)=O\l(\lambda^{-2a+4}\r);\\
     \frac{f_\gamma^2(\lambda,\,\theta)}{f^3(\lambda,\,\theta)}w(\lambda)& =\frac{32\pi\gamma^2}\beta\l(\lambda^2-\alpha^2-\gamma^2\r)^2s^{-1}(\lambda)w(\lambda)=O\l(\lambda^{-2a}\r).
   \end{aligned}
  \label{ordr_phis^2/f^3w4exle}
 \end{equation}
Therefore the continuous in $(\lambda,\theta)\in\mathbb{R}\times\Theta^c$ functions \eqref{ordr_phis^2/f^3w4exle} are bounded in $(\lambda,\theta)\in\mathbb{R}\times\Theta^c$, if $a\ge2$.

Using equations \eqref{ordr_phi*phi/f^2w4exle} and \eqref{ordr_2nd_der_f/f*w4exle} we obtain uniformly in $\theta\in\Theta^c$, as $\lambda\to\infty$,
\begin{equation}
   \begin{aligned}
     \frac{f_{\alpha\alpha}(\lambda,\,\theta)}{f^2(\lambda,\,\theta)}w(\lambda)&=-\frac{8\pi}\beta\l(\lambda^2+3\alpha^2+\gamma^2\r)w(\lambda) +\frac{16\pi\alpha}\beta\l(\lambda^2+\alpha^2+\gamma^2\r)s^{-1}(\lambda)s_\alpha'(\lambda)w(\lambda) =O\l(\lambda^{-2a+2}\r);\\
     \frac{f_{\beta\beta}(\lambda,\,\theta)}{f^2(\lambda,\,\theta)}w(\lambda)&=0;\\
     \frac{f_{\gamma\gamma}(\lambda,\,\theta)}{f^2(\lambda,\,\theta)}w(\lambda)&=\frac{8\pi}\beta\l(\lambda^2-\alpha^2-3\gamma^2\r)w(\lambda) -\frac{16\pi\gamma}\beta\l(\lambda^2-\alpha^2-\gamma^2\r)s^{-1}(\lambda)s_\gamma'(\lambda)w(\lambda) =O\l(\lambda^{-2a+2}\r);\\
     \frac{f_{\alpha\beta}(\lambda,\,\theta)}{f^2(\lambda,\,\theta)}w(\lambda)&=-\frac{8\pi\alpha}{\beta^2} \l(\lambda^2+\alpha^2+\gamma^2\r)w(\lambda) =O\l(\lambda^{-2a+2}\r);\\
     \frac{f_{\alpha\gamma}(\lambda,\,\theta)}{f^2(\lambda,\,\theta)}w(\lambda)&=-\frac{16\alpha\gamma}\beta w(\lambda)+ \frac{32\pi\alpha\gamma}\beta\l(\lambda^4-\l(\alpha^2+\gamma^2\r)^2\r)s^{-1}(\lambda)w(\lambda) =O\l(\lambda^{-2a}\r);\\
     \frac{f_{\beta\gamma}(\lambda,\,\theta)}{f^2(\lambda,\,\theta)}w(\lambda)&=\frac{8\pi\gamma}{\beta^2}\l(\lambda^2-\alpha^2-\gamma^2\r)w(\lambda) =O\l(\lambda^{-2a+2}\r).
   \end{aligned}
  \label{ordr_2nd_der_f/f^2*w4exle}
 \end{equation}
So, continuous on  $\mathbb{R}\times\Theta^c$ functions \eqref{ordr_2nd_der_f/f^2*w4exle} are bounded in $(\lambda,\theta)\in\mathbb{R}\times\Theta^c$, if $a\ge1$.

To check \textbf{N$_5$(ii)} consider the weight function
 \[
   v(\lambda)=\l(1+\lambda^2\r)^{-b},\ \lambda\in\mathbb{R},\ b>0.
 \]
If $a\ge b$, then function $\frac{w(\lambda)}{v(\lambda)}$ is bounded on $\mathbb{R}$ (condition \textbf{N$_5$(iii)}). Using \eqref{ordr_phis^2/f^3w4exle} we obtain uniformly in $\theta\in\Theta^c$, as $\lambda\to\infty$,
 \begin{equation}
   \begin{aligned}
     \frac{f_\alpha^2(\lambda,\,\theta)}{f^3(\lambda,\,\theta)}v(\lambda)& =\frac{16\pi\alpha^2}\beta\l(\lambda^2+\alpha^2+\gamma^2\r)^2s^{-1}(\lambda)v(\lambda)=O\l(\lambda^{-2b}\r);\\
     \frac{f_\beta^2(\lambda,\,\theta)}{f^3(\lambda,\,\theta)}v(\lambda)& =\frac{2\pi}{\beta^3}s(\lambda)v(\lambda)=O\l(\lambda^{-2b+4}\r);\\
     \frac{f_\gamma^2(\lambda,\,\theta)}{f^3(\lambda,\,\theta)}v(\lambda)& =\frac{32\pi\gamma^2}\beta\l(\lambda^2-\alpha^2-\gamma^2\r)^2s^{-1}(\lambda)v(\lambda)=O\l(\lambda^{-2b}\r).
   \end{aligned}
  \label{ordr_der_f^2/f^3v4exle}
 \end{equation}

In turn, similarly to \eqref{ordr_phi*phi/f^2w4exle} it follows uniformly in $\theta\in\Theta^c$, as $\lambda\to\infty$,
 \begin{equation}
   \begin{aligned}
     \frac{f_\alpha(\lambda,\,\theta)f_\beta(\lambda,\,\theta)}{f^3(\lambda,\,\theta)}v(\lambda)& =-\frac{4\pi\alpha}{\beta^2}\l(\lambda^2+\alpha^2+\gamma^2\r)v(\lambda)=O\l(\lambda^{-2b+2}\r);\\
     \frac{f_\alpha(\lambda,\,\theta)f_\gamma(\lambda,\,\theta)}{f^3(\lambda,\,\theta)}v(\lambda)& =-\frac{16\alpha\gamma}\beta \l(\lambda^4-\l(\alpha^2+\gamma^2\r)^2\r)s^{-1}(\lambda)v(\lambda)=O\l(\lambda^{-2b}\r);\\
     \frac{f_\beta(\lambda,\,\theta)f_\gamma(\lambda,\,\theta)}{f^3(\lambda,\,\theta)}v(\lambda)& =\frac{8\pi\gamma}{\beta^2}\l(\lambda^2-\alpha^2-\gamma^2\r)v(\lambda)=O\l(\lambda^{-2b+2}\r).
   \end{aligned}
  \label{ordr_ders_f*f/f^3v4exle}
 \end{equation}

The functions \eqref{ordr_der_f^2/f^3v4exle} and \eqref{ordr_ders_f*f/f^3v4exle} will be uniformly continuous in $(\lambda,\theta)\in\mathbb{R}\times\Theta^c$, if they converge to zero, as $\lambda\to\infty$, uniformly in $\theta\in\Theta^c$, that is if $b>2$.

Similarly to \eqref{ordr_2nd_der_f/f^2*w4exle} uniformly in $\theta\in\Theta^c$, as $\lambda\to\infty$,
\begin{equation}
   \begin{aligned}
     \frac{f_{\alpha\alpha}(\lambda,\,\theta)}{f^2(\lambda,\,\theta)}v(\lambda)&=O\l(\lambda^{-2b+2}\r);\ \
     \frac{f_{\beta\beta}(\lambda,\,\theta)}{f^2(\lambda,\,\theta)}v(\lambda)&=&0;&\ \
     \frac{f_{\gamma\gamma}(\lambda,\,\theta)}{f^2(\lambda,\,\theta)}v(\lambda)&=O\l(\lambda^{-2b+2}\r);\\
     \frac{f_{\alpha\beta}(\lambda,\,\theta)}{f^2(\lambda,\,\theta)}v(\lambda)&=O\l(\lambda^{-2b+2}\r);\ \
     \frac{f_{\alpha\gamma}(\lambda,\,\theta)}{f^2(\lambda,\,\theta)}v(\lambda)&=&O\l(\lambda^{-2b}\r);&\ \
     \frac{f_{\beta\gamma}(\lambda,\,\theta)}{f^2(\lambda,\,\theta)}v(\lambda)& =O\l(\lambda^{-2b+2}\r).
   \end{aligned}
  \label{ordr_2nd_der_f/f^2*v4exle}
 \end{equation}
Thus the functions \eqref{ordr_der_f^2/f^3v4exle}--\eqref{ordr_2nd_der_f/f^2*v4exle} are uniformly continuous in $(\lambda,\theta)\in\mathbb{R}\times\Theta^c$, if $b>2$.

Proceeding to the verification of condition \textbf{N$_6$}, we note that for any $x=\l(x_\alpha,\,x_\beta,\,x_\gamma\r)\ne0$
 \[
   \l<W_1(\theta)x,\,x\r>=\int\limits_{\mathbb{R}}\,\l(x_\alpha f_\alpha(\lambda,\,\theta)+x_\beta f_\beta(\lambda,\,\theta) +x_\gamma f_\gamma(\lambda,\,\theta)\r)\frac{w(\lambda)}{f^2(\lambda,\,\theta)}d\lambda.
 \]
From equation \eqref{drvtv_spec_dens4exle} it is seen that the positive definiteness of the matrix $W_1(\lambda)$ follows from linear independence of the functions $\lambda^2+\alpha^2+\gamma^2$, $s(\lambda)$, $\lambda^2-\alpha^2-\gamma^2$. Positive definiteness of the matrix $W_2(\theta)$ is established similarly.

In our example to satisfy the consistency conditions \textbf{C$_4$} and \textbf{C$_5$} the weight functions $w(\lambda)$ and $v(\lambda)$ should be chosen so that $a\ge b>2$. On the other hand to satisfy the asymptotic normality conditions \textbf{N$_4$} and \textbf{N$_5$} the functions $w(\lambda)$ and $v(\lambda)$ should be such that $a>\frac52$ and $a\ge b>2$.

The spectral density \eqref{spec_dens_eps_repr_exle} has no singularity at zero, so that the functions $v(\lambda)$ in the conditions \textbf{C$_5$(i)} and \textbf{N$_5$(ii)} could be chosen to be equal to $w(\lambda)$, for example, $a=b=3$. However we prefer to keep in the text the function $v(\lambda)$, since it is needed when the spectral density could have a singularity at zero or elsewhere, see, e.g., Example 1~\cite{LeoSa_oWE4SCoCPRP}, where linear process driven by the Brownian motion and regression function $g(t,\,\alpha)\equiv0$ have been studied. Specifically in the case of Riesz-Bessel spectral density
 \begin{equation}
    f(\lambda,\,\theta)=\frac{\beta}{2\pi|\lambda|^{2\alpha}(1+\lambda^2)^\gamma},\ \lambda\in\mathbb{R},
 \label{Rie_Bess_sp_fn}
 \end{equation}
where $\theta=\l(\theta_1,\,\theta_2,\,\theta_3\r)=(\alpha,\,\beta,\,\gamma)\in\Theta=(\underline{\alpha},\,\overline{\alpha}) \times(\underline{\beta},\,\overline{\beta})\times(\underline{\gamma},\,\overline{\gamma})$, $\underline{\alpha}>0$, $\overline{\alpha}<\frac12$, $\underline{\beta}>0$, $\overline{\beta}<\infty$, $\underline{\gamma}>\frac12$, $\overline{\gamma}<\infty$, and the parameter $\alpha$ signifies the long range dependence, while the parameter $\gamma$ indicates the second-order intermittency~\cite{AnhLeoSa_oCoMCE4FSP,Gaoetal_PEoSPwLRD&I,LimTeo_SPPoFRBFoVO}, the weight functions have been chosen in the form
 \[
   w(\lambda)=\frac{\lambda^{2b}}{\l(1+\lambda^2\r)^a},\ a>b>0;\ \ \ v(\lambda)=\frac{\lambda^{2b'}}{\l(1+\lambda^2\r)^{a'}},\ a'>b'>0,\ \lambda\in\mathbb{R}.
 \]

Unfortunately, our conditions do not cover so far the case of the general non-linear regression function and Lévy driven continuous-time strongly dependent linear random noise such as Riesz-Bessel motion.

\begin{appendices}
\numberwithin{lemma}{section}
\numberwithin{theorem}{section}
\numberwithin{definition}{section}
\numberwithin{example}{section}
\section{LSE consistency}\label{app_LSE_cons}

$\indent$Some results on consistency of the LSE $\widehat{\alpha}_T$ in the observation model of the type \eqref{c_n_reg_m} with stationary noise $\varepsilon(t)$, $t\in\mathbb{R}$, were obtained, for example, in Ivanov and Leonenko~\cite{IvLeo_SAoRF_En,IvLeo_AToNRwLRD,IvLeo_REiNRMwLRD,IvLeo_SAoLRDiNR}, Ivanov~\cite{Iv_aSotPoDHP,Iv_CoLSEoAaAFoSoHO}, Ivanov et al.~\cite{IvLeRuMeZhu_EoHCiRwCDE} to mention several of the relevant works. In this section we formulate a generalization of Malinvaud theorem~\cite{Mlvd_tCoNR} on $\widehat{\alpha}_T$ consistency for linear stochastic process \eqref{LinRep_RmdnNse} and consider an example of nonlinear regression function $g(t,\,\alpha)$ satisfying the conditions of this theorem and conditions \textbf{C$_1$}, \textbf{C$_2$}. Then we consider another possibilities of \textbf{C$_1$} and \textbf{C$_2$} fulfillment.

Set
 \[
  \begin{aligned}
   w_T(\alpha_1,\,\alpha_2)&=\int\limits_0^T\,\varepsilon(t)\l(g(t,\,\alpha_1)-g(t,\,\alpha_2)\r)dt,\ \alpha_1,\alpha_2\in\mathcal{A}^c,\\
   \Psi_T(u_1,\,u_2)&=\Phi_T\l(\alpha_0+T^{\frac12}d_T^{-1}(\alpha_0)u_1,\ \alpha_0+T^{\frac12}d_T^{-1}(\alpha_0)u_1\r).
  \end{aligned}
 \]

For any fixed $\alpha_0\in\mathcal{A}$, the function\ \ $\Psi_T(u_1,\,u_2)$\ \ is defined on the set $U_T(\alpha_0)\times U_T(\alpha_0)$, $U_T(\alpha_0)=T^{-\frac12}d_T(\alpha_0)\l(\mathcal{A}^c-\alpha_0\r)$.

Assume the following.

\textbf{1)} For any $\varepsilon>0$ and $R>0$ there exists $\delta=\delta(\varepsilon,\,R)$ such that
 \begin{equation}
    \sup\limits_{\substack{u_1,u_2\in U_T(\alpha_0)\cap v^c(R)\\ \|u_1-u_2\|\le\delta}}\,T^{-1}\Psi_T(u_1,\,u_2)\le\varepsilon.
  \label{sup_T^-1Psi_T<=eps}
 \end{equation}

\textbf{2)} For some $R_0>0$ and any $\rho\in(0,\,R_0)$ there exist numbers $a=a(R_0)>0$ and $b=b(\rho,\,R_0)$ such that
 \begin{equation}
    \inf\limits_{u\in U_T(\alpha_0)\cap\l(v^c(R_0)\backslash v(\rho)\r)}\, T^{-1}\Psi(u,\,0)\ge b;
  \label{inf_T^-1Psi_T>=b}
 \end{equation}
 \begin{equation}
    \inf\limits_{u\in U_T(\alpha_0)\backslash v^c(R_0)}\, T^{-1}\Psi(u,\,0)\ge 4B(0)+a.
  \label{inf_T^-1Psi_T>=4B0+a}
 \end{equation}

It was proven in Lemma \ref{lema_int_eps^2} that under condition \textbf{A$_1$}
 \begin{equation}
     \ExpV\l(\nu_T^*-B(0)\r)^2=O\l(T^{-1}\r).
  \label{ordrE_nu_T-B0}
 \end{equation}

\begin{lemma}
 Under condition \textbf{A$_1$},
  \begin{equation}
      \ExpV w_T^4(\alpha_1,\,\alpha_2)\le c \Phi_T^2(\alpha_1,\,\alpha_2),\ \alpha_1,\alpha_2\in\mathcal{A}^c.
   \label{Ew^4<=cPhi_T^2}
  \end{equation}
\end{lemma}

\begin{proof}
 By formula \eqref{4ordr_mmnt_fn}
  \[
    \begin{aligned}
       \ExpV& w_T^4(\alpha_1,\,\alpha_2)=\int\limits_{[0,T]^4}\,c_4(t_1,\,t_2,\,t_3,\,t_4)\prod\limits_{i=1}^4\, \l(g(t_i,\,\alpha_1)-g(t_i,\,\alpha_2)\r)dt_1dt_2dt_3dt_4+\\
       &+3\l(\int\limits_0^T\int\limits_0^T\,B(t_1-t_2)\l(g(t_1,\,\alpha_1)-g(t_1,\,\alpha_2)\r) \l(g(t_2,\,\alpha_1)-g(t_2,\,\alpha_2)\r)dt_1dt_2\r)^2=I_7+3I_8^2.
    \end{aligned}
  \]

 By condition \textbf{A$_1$} and Fubini-Tonelli theorem
  \[
    \begin{aligned}
       |I_8|&\le\frac12\int\limits_0^T\int\limits_0^T\,|B(t_1-t_2)|\l[\l(g(t_1,\,\alpha_1)-g(t_1,\,\alpha_2)\r)^2 +\l(g(t_2,\,\alpha_1)-g(t_2,\,\alpha_2)\r)^2\r]dt_1dt_2\le
       d_2\l\|\hat{a}\r\|_1^2,
    \end{aligned}
  \]
 $\l\|\hat{a}\r\|_1=\int\limits_{\mathbb{R}}\,|\hat{a}(t)|dt$.

 On the other hand by formula \eqref{rep_c_r_by_hat-a}
  \[
     \begin{aligned}
       |I_7|&\le d_4\int\limits_{\mathbb{R}}\,ds\int\limits_{[0,T]^4}\,\prod\limits_{i=1}^4\, \Bigl|\hat{a}(t_i-s)\l(g(t_i,\,\alpha_1)-g(t_i,\,\alpha_2)\r)\Bigr|dt_1dt_2dt_3dt_4\le\\
       &\le\frac12 d_4\int\limits_{\mathbb{R}}\,ds\int\limits_{[0,T]^4}\,\prod\limits_{i=1}^4\, \l|\hat{a}(t_i-s)\r|\Bigl[\bigl(g(t_1,\,\alpha_1)-g(t_1,\,\alpha_2)\bigr)^2 \bigl(g(t_2,\,\alpha_1)-g(t_2,\,\alpha_2)\bigr)^2+\Bigr.\\
       &\hspace*{27mm}\Bigl.+\bigl(g(t_3,\,\alpha_1)-g(t_3,\,\alpha_2)\bigr)^2 \bigl(g(t_4,\,\alpha_1)-g(t_4,\,\alpha_2)\bigr)^2\Bigr]dt_1dt_2dt_3dt_4=I_7^{(1)}+I_7^{(2)};
     \end{aligned}
  \]
  \[
    \begin{aligned}
      I_7^{(1)}=&\frac12 d_4\int\limits_{\mathbb{R}}\,ds\int\limits_0^T\int\limits_0^T\,\, \bigl|\hat{a}(t_1-s)\hat{a}(t_2-s)\bigr|\bigl(g(t_1,\,\alpha_1)-g(t_1,\,\alpha_2)\bigr)^2 \bigl(g(t_2,\,\alpha_1)-g(t_2,\,\alpha_2)\bigr)^2dt_1dt_2\ \cdot\\
      &\hspace*{87mm}\cdot \int\limits_0^T\int\limits_0^T\,\, \bigl|\hat{a}(t_3-s)\hat{a}(t_4-s)\bigr|dt_3dt_4\le\\
      \le&\frac14 d_4\l\|\hat{a}\r\|_1^2\int\limits_0^T\int\limits_0^T\,\, \bigl(g(t_1,\,\alpha_1)-g(t_1,\,\alpha_2)\bigr)^2 \bigl(g(t_2,\,\alpha_1)-g(t_2,\,\alpha_2)\bigr)^2\ \cdot\\
      &\hspace*{72mm}\cdot \l(\int\limits_{\mathbb{R}}\,\Bigl[\hat{a}^2(t_1-s)+\hat{a}^2(t_2-s)\Bigr]ds\r)dt_1dt_2\le\\
      \le& \frac12 d_4\l\|\hat{a}\r\|_1^2\Phi_T^2(\alpha_1,\,\alpha_2).\hspace*{96mm}
    \end{aligned}
  \]

 For integral $I_7^{(2)}$ we get the same bound. So, we obtain inequality \eqref{Ew^4<=cPhi_T^2} with
  \[
    c=d_4\l\|\hat{a}\r\|_1^2\l\|\hat{a}\r\|_2^2+3d_2^2\l\|\hat{a}\r\|_1^4.
  \]
\end{proof}

\begin{theorem}
  If assumptions \textbf{1), 2),} and \textbf{A$_1$} are valid then for any $\rho>0$
  \[
     \Prob\l\{\l\|T^{-\frac12}d_T(\alpha_0)\l(\widehat{\alpha}_T-\alpha_0\r)\r\|\ge\rho\r\}=O(T^{-1}),\ \ \text{as}\ \ T\rightarrow\infty.
  \]
\end{theorem}

\begin{prfapthm}
  The proof of this Malinvaud theorem generalization is similar to the proof of Theorem 3.2.1. in \cite{IvLeo_SAoRF_En} and uses the relations \eqref{ordrE_nu_T-B0} and \eqref{Ew^4<=cPhi_T^2}.
\end{prfapthm}

Instead of \textbf{C$_2$} consider the stronger condition.

\textbf{C$_2'$}. There exist positive constants $c_0,\,c_1<\infty$ such that for any $\alpha\in\mathcal{A}^c$ and $T>T_0$
  \begin{equation}
     c_1\bigl\|d_T(\alpha)\l(\alpha_1-\alpha_2\r)\bigr\|^2\ \le\ \Phi_T(\alpha_1,\,\alpha_2)\ \le\  c_0\bigl\|d_T(\alpha)\l(\alpha_1-\alpha_2\r)\bigr\|^2,\ \alpha_1,\alpha_2\in\mathcal{A}^c.
   \label{dbnd4Phi_T_w_d_app}
  \end{equation}

Point out a sufficient condition for \textbf{C$_2'$} fulfillment. Introduce a diagonal matrix
 \[
   s_T=\diag\Bigl(s_{iT},\ i=\overline{1,q}\Bigr),\ s_{iT}\to\infty,\ \text{as}\ T\rightarrow\infty,\ i=\overline{1,q}.
 \]

\textbf{C$_2''$. (i)} There exist positive constants $\underline{c}_i,\ \overline{c}_i$, $i=\overline{1,q}$, such that for $T>T_0$ uniformly in $\alpha\in\mathcal{A}$
 \begin{equation}
     \underline{c}_i<s_{iT}^{-1}d_{iT}(\alpha)<\overline{c}_i,\ i=\overline{1,q}.
  \label{dbnd4s_iT&d_iT_app}
 \end{equation}
\hspace*{12mm}\textbf{(ii)} For some numbers $c_0^*.\ c_1^*$ and $T>T_0$,
 \[
    c_0^*\,\bigl\|s_T\l(\alpha_1-\alpha_2\r)\bigr\|^2\le \Phi_T(\alpha_1,\,\alpha_2)\le c_0^*\,\bigl\|s_T\l(\alpha_1-\alpha_2\r)\bigr\|^2,\ \alpha_1,\alpha_2\in\mathcal{A}^c.
 \]

Under condition \textbf{C$_2''$} as it is easily seen one can take in \textbf{C$_2'$}
 \[
   c_0=c_0^*\l(\min\limits_{1\le i \le q}\,\underline{c}_i\r)^{-1},\ \ \ c_1=c_1^*\l(\max\limits_{1\le i \le q}\,\overline{c}_i\r)^{-1}.
 \]

The next example demonstrates the fulfillment of the condition \textbf{C$_2'$} (compare with Ivanov and Orlovskyi~\cite{IvOr_LDoRPEiCTMwsGN}).

\begin{example}
 Let
  \[
    g(t,\,\alpha)=\exp\l\{\l<\alpha,\,y(t)\r>\r\},
  \]
 with $\l<\alpha,\,y(t)\right>=\sum\limits_{i=1}^q\,\alpha_iy_i(t)$, regressors $y(t)=\Bigl(y_1(t),\,\ldots,\,y_q(t)\Bigr)'$, $t\ge0$, take values in a compact set $Y\subset\mathbb{R}^q$. Suppose
  \[
     J_T=\l(T^{-1}\int\limits_0^Ty_i(t)y_j(t)dt\r)_{i,j=1}^q\ \to\ J=\l(J_{ij}\r)_{i,j=1}^q,\ \ \text{as}\ \ T\to\infty,
  \]
 where $J$ is a positive definite matrix, and the set $\mathcal{A}$ in the model \eqref{c_n_reg_m} is bounded. Set
  \[
    M=\underset{\alpha\in\mathcal{A}^c,\,y\in Y}\max\,\exp\l\{\l<\alpha,\,y\r>\r\},\ L=\underset{\alpha\in\mathcal{A}^c,\, y\in Y}\min\,\exp\l\{\l<\alpha,\,y\r>\r\}.
  \]
 Then for any $\delta>0$ and $T>T_0$
  \[
     L^2\l(J_{ii}-\delta\r)< T^{-1} d_{iT}^2(\alpha)< M^2\l(J_{ii}+\delta\r),\ i=\overline{1,q},
  \]
 and condition \textbf{C$_2''$(i)} is fulfilled with matrix $s_T=T^{\frac12}\mathbb{I}_q$, $\mathbb{I}_q$ is identity matrix of order $q$, and $\underline{c}_i=L^2\l(J_{ii}-\delta\r)$, $\overline{c}_i=M^2\l(J_{ii}+\delta\r)$, $i=\overline{1,q}$.

 Let us check the condition \textbf{C$_2''$(ii)}. We have
  \[
     e^{\l<\alpha_1,\,y(t)\r>}-e^{\l<\alpha_2,\,y(t)\r>} =e^{\l<\alpha_2,\,y(t)\r>}\l(e^{\l<\alpha_1-\alpha_2,\,y(t)\r>}-1\r).
  \]
 As far as $\left(e^x-1\right)^2\ge x^2$, $x\ge0$, and $\left(e^x-1\right)^2\ge e^{2x}x^2$, $x<0$, then
  \[
    \l(e^{\l<\alpha_1-\alpha_2,\,y(t)\r>}-1\r)^2\ge\Delta\,\l<\alpha_1-\alpha_2,\,y(t)\r>^2,\ \ \Delta=\min\l\{1,\ e^{2\l<\alpha_1-\alpha_2,\,y(t)\r>}\r\}.
  \]
 Thus
  \[
    e^{2\l<\alpha_2,\,y(t)\r>}\l(e^{\l<\alpha_1-\alpha_2,\,y(t)\r>}-1\r)^2\ge e^{2\l<\alpha_2,\,y(t)\r>}\Delta\,\l<\alpha_1-\alpha_2,\,y(t)\r>^2\ge L^2\l<\alpha_1-\alpha_2,\,y(t)\r>^2,
  \]
 and for any $\delta>0$ and $T>T_0$
  \[
    \Phi_T(\alpha_1,\,\alpha_2)\ge L_2\sum\limits_{i,j=1}^q\,J_{ij,T}\l(T^{\frac12}(\alpha_{i1}-\alpha_{i2})\r) \l(T^{\frac12}(\alpha_{j1}-\alpha_{j2})\r)\ge L^2\bigl(\lambda_{\min}(J)-\delta\bigr) \l\|T^{\frac12}(\alpha_1-\alpha_2)\r\|^2,
  \]
 where $\lambda_{\min}(J)$ is the least eigenvalue of the matrix $J$.

 On the other hand,
  \[
    \l|e^{\l<\alpha_1,\,y(t)\r>}-e^{\l<\alpha_2,\,y(t)\r>}\r|=\l|\sum\limits_{i=1}^q\,y_i(t) e^{\l<y(t),\,\alpha_1+\eta(t)(\alpha_2-\alpha_1)\r>}(\alpha_{i1}-\alpha_{i2})\r|\le M\l|\sum\limits_{i=1}^q\,y_i(t) (\alpha_{i1}-\alpha_{i2})\r|,
  \]
 $\eta(t)\in(0,\,1)$, and
  \[
    \Phi_T(\alpha_1,\,\alpha_2)\le M^2\int\limits_0^T\,\l(\sum\limits_{i=1}^q\,y_i(t) (\alpha_{i1}-\alpha_{i2})\r)^2dt \le M^2\bigl(\lambda_{\max}(J)+\delta\bigr) \l\|T^{\frac12}(\alpha_1-\alpha_2)\r\|^2,
  \]
 where $\lambda_{\max}(J)$ is the maximal eigenvalue of the matrix $J$. It means that condition \textbf{C$_2''$(ii)} is valid for matrix $s_T=T^{\frac12}\mathbb{I}_q$.

 So the condition \textbf{C$_2'$} is valid as well and in \eqref{dbnd4Phi_T_w_d_app} one can choose for $T>T_0$ some numbers
  \[
    c_0>\frac{M^2\lambda_{\max}(J)}{L^2\min\limits_{1\le i\le q}\,J_{ii}},\ \ \ c_1<\frac{L^2\lambda_{\min}(J)}{M^2\max\limits_{1\le i\le q}\,J_{ii}}.
  \]
\end{example}

Inequalities \eqref{dbnd4Phi_T_w_d_app} can be rewritten in the equivalent form
 \begin{equation}
    c_1\bigl\|u-v\bigr\|^2\ \le\ T^{-1}\Psi_T(u,\,v)\ \le\  c_0\bigl\|u-v\bigr\|^2,\ u,v\in U_T(\alpha),\ \alpha\in\mathcal{A}.
  \label{dbnd4Psi_T_app}
 \end{equation}
From the right hand side of \eqref{dbnd4Psi_T_app} it follows \eqref{sup_T^-1Psi_T<=eps}. Similarly, from the left hand side of \eqref{dbnd4Psi_T_app} taking $\nu=0$ we obtain \eqref{inf_T^-1Psi_T>=b} for any $R_0>0$ and it is possible to choose $R_0>0$ satisfying \eqref{inf_T^-1Psi_T>=4B0+a}.

In our example \textbf{A$_1$} due to inequalities \eqref{dbnd4s_iT&d_iT_app} with $s_{iT}=T^{\frac12}$, $i=\overline{1,q}$, the set $U_T(\alpha)$  is bounded uniformly in $T$ and it is not necessary to use condition \eqref{inf_T^-1Psi_T>=4B0+a}. However in Malinvaud theorem we can not ignore the condition \eqref{inf_T^-1Psi_T>=4B0+a} of parameters distinguishability in the cases when the sets $U_T(\alpha)$ expands to infinity as $T\to\infty$ or the set $\mathcal{A}$ is unbounded.

It goes without saying not all the interesting classes of nonlinear regression functions satisfy consistency conditions of Malinvaud or, say, Jennrich~\cite{Jen_APoNLSE} types. The important example of such a class is given by the trigonometric regression functions.

\begin{example}\label{exmp_trig_reg_fn}
 Let
  \begin{equation}
     g(t,\,\alpha)=\sum\limits_{i=1}^N\,\l(A_i\cos\varphi_it+B_i\sin\varphi_it\r),
   \label{trig_reg_fn_app}
  \end{equation}
 $\alpha=\l(\alpha_1,\,\alpha_2,\,\alpha_3,\,\ldots,\,\alpha_{3N-2},\,\alpha_{3N-1},\,\alpha_{3N}\r) =\l(A_1,\,B_1,\,\varphi_1,\,\ldots,\,A_N,\,B_N,\,\varphi_N\r)$, $0<\underline{\varphi}<\varphi_1<\ldots<\varphi_N<\overline{\varphi}<\infty$.

 Under some conditions on angular frequencies $\varphi=\l(\varphi_1,\,\ldots,\,\varphi_N\r)$ distinguishability\ \  (see Walker~\cite{Wal_oEoHCiTSwSDR}, Ivanov~\cite{Iv_aSotPoDHP}, Ivanov et al~\cite{IvLeRuMeZhu_EoHCiRwCDE}) it is possible to prove that at least
  \begin{equation}
     T^{-1}\Phi_T(\widehat{\alpha}_T,\,\alpha_0)\ \overset{\Prob}\longrightarrow\ 0,\ \ \text{as}\ \ T\to\infty,
   \label{cnv_T^-1Phi_T20_ex-app}
  \end{equation}
 $\widehat{\alpha}_T=\l(A_{1T},\,B_{1T},\,\varphi_{1T},\,\ldots,\,A_{NT},\,B_{NT},\,\varphi_{NT}\r)$, $\alpha_0=\l(A_1^0,\,B_1^0,\,\varphi_1^0,\,\ldots,\,A_N^0,\,B_N^0,\,\varphi_N^0\r)$, $\l(C_k^0\r)^2=\l(A_k^0\r)^2+\l(B_k^0\r)^2>0$, $k=\overline{1,N}$.

 The convergence in \eqref{cnv_T^-1Phi_T20_ex-app} can be a.s. In turn, from \eqref{cnv_T^-1Phi_T20_ex-app} it follows (see cited papers)
  \begin{equation}
     A_{iT}\ \overset{\Prob}\longrightarrow\ A_i^0,\ \ B_{iT}\ \overset{\Prob}\longrightarrow\ B_i^0,\ \ T\l(\varphi_{iT}-\varphi_i^0\r)\ \overset{\Prob}\longrightarrow\ 0,\ \ \text{as}\ \ T\to\infty.
   \label{Pcons4A_B_phi_ex-app}
  \end{equation}
 Note that
  \begin{equation}
     T^{-1}d^2_{3k-2,T}(\alpha_0),\ T^{-1}d^2_{3k-1,T}(\alpha_0)\ \to\ \frac12,\ T^{-3}d^2_{3k,T}(\alpha_0)\ \to\ \frac16\l(\l(A_k^0\r)^2+\l(B_k^0\r)^2\r),\ \ \text{as}\ \ T\to\infty,
   \label{cnv4T^-1d_T_ex-app}
  \end{equation}
 $k=\overline{1,N}$.

 From \eqref{Pcons4A_B_phi_ex-app} and \eqref{cnv4T^-1d_T_ex-app} we obtain the relation of condition \textbf{C$_1$} for trigonometric regression:
  \[
    T^{-\frac12}d_T(\alpha_0)\l(\widehat{\alpha}_T-\alpha_0\r)\ \overset{\Prob}\longrightarrow\ 0,\ \ \text{as}\ \ T\to\infty.
  \]
 To check the fulfillment of the condition \textbf{C$_2$} for regression function \eqref{trig_reg_fn_app} we get
  \begin{equation}
     \begin{aligned}
     \l|A_i\cos\varphi_it+B_i\sin\varphi_it-A_i^0\cos\varphi_it-B_i^0\sin\varphi_it\r|\le
     \l|A_i-A_i^0\r|+\l|B_i-B_i^0\r|+\l(|A_i^0|+|B_i^0|\r)t\l|\varphi_i-\varphi_i^0\r|,
     \end{aligned}
   \label{inq4dif_trig_comp_ex-app}
  \end{equation}
 $k=\overline{1,N}$, and therefore
  \[
    \Phi_T(\widehat{\alpha}_T,\,\alpha_0)\le3N\sum\limits_{i=1}^N\, \l(T\l(A_i-A_i^0\r)^2+T\l(B_i-B_i^0\r)^2+\frac13\l(|A_i^0|+|B_i^0|\r)^2T^3\l(\varphi_i-\varphi_i^0\r)^2\r).
  \]
 Using again the relations \eqref{cnv4T^-1d_T_ex-app} we arrive at the inequality of the condition \textbf{C$_2$}.
  \begin{equation}
     \Phi_T(\alpha,\,\alpha_0)\ \le\  c_0\bigl\|d_T(\alpha_0)\l(\alpha-\alpha_0\r)\bigr\|^2,\ \alpha\in\mathcal{A}^c.
   \label{inq4Phi_T_l_C2_ex-app}
  \end{equation}
 with constant $c_0$ depending on $A_i^0$, $B_i^0$, $i=\overline{1,N}$.
\end{example}

The next lemma is the main part of the convergence \eqref{cnv_T^-1Phi_T20_ex-app} proof.

\begin{lemma}
 Under condition \textbf{A$_1$}
  \begin{equation}
     \xi(T)=\sup\limits_{\lambda\in\mathbb{R}}\,\l|T^{-1}\int\limits_0^T\,e^{-i\lambda t}\varepsilon(t)dt\r|\ \overset{\Prob}\longrightarrow\ 0,\ \ \text{as}\ \ T\to\infty.
   \label{cnv4supT^-1int_e_eps_app}
  \end{equation}
\end{lemma}

\begin{proof}
 Since
  \[
    \l|\int\limits_0^T\,e^{-i\lambda t}\varepsilon(t)dt\r|^2=\int\limits_{-T}^T\,e^{-i\lambda u} \int\limits_0^{T-|u|}\,\varepsilon(t+|u|)\varepsilon(t)dtdu=2\int\limits_0^T\,\cos{\lambda u} \int\limits_0^{T-u}\,\varepsilon(t+u)\varepsilon(t)dtdu,
  \]
 then
  \[
    E\xi^2(T)\le 2T^{-2}\int\limits_0^T\,\ExpV\l|\int\limits_0^{T-u}\,\varepsilon(t+u)\varepsilon(t)dt\r|du \le 2T^{-2}\int\limits_0^T\,K^{\frac12}(u)du.
  \]

 By formula \eqref{4ordr_mmnt_fn}
  \[
    \begin{aligned}
       K(u)=&\int\limits_0^{T-u}\int\limits_0^{T-u}\,\ExpV\varepsilon(t+u)\varepsilon(s+u)\varepsilon(t)\varepsilon(s)dtds =\int\limits_0^{T-u}\int\limits_0^{T-u}\,c_4(t+u,\,s+u,\,t,\,s)dtds+\\
       &+(T-u)^2B^2(u)+\int\limits_0^{T-u}\int\limits_0^{T-u}\,B^2(t-s)dtds +\int\limits_0^{T-u}\int\limits_0^{T-u}\,B(t-s+u)B(t-s-u)dtds\le\\
       &\le K_1(u)+K_2(u)+K_3(u)+|K_4(u)|,
    \end{aligned}
  \]
 and
  \begin{equation}
     E\xi^2(T)\le 2T^{-2}\int\limits_0^T\,\l(K_1^{\frac12}(u)+K_2^{\frac12}(u)+K_3^{\frac12}(u)+\l|K_4(u)\r|^{\frac12}\r)du.
   \label{inq4Exi^2T_app}
  \end{equation}

 By formula \eqref{rep_c_r_by_hat-a}
  \[
   \begin{aligned}
     K_1(u)=&d_4\int\limits_{\mathbb{R}}\,\l(\int\limits_0^{T-u}\,\hat{a}(t+u-r)\hat{a}(t-r)\r)^2dr
     \le d_4\int\limits_{\mathbb{R}}\,\l(\int\limits_0^{T-u}\,\hat{a}^2(t+u-r)dt \int\limits_0^{T-u}\,\hat{a}^2(t-r)dt\r)dr\le\\
     &\le d_4\l\|\hat{a}\r\|_2^2\int\limits_0^{T-u}\,dt\int\limits_{\mathbb{R}}\,\hat{a}^2(t+u-r)dr \le d_4\l\|\hat{a}\r\|_2^4(T-u),
   \end{aligned}
  \]
 that is
  \begin{equation}
     T^{-2}\int\limits_0^T\,K_1^{\frac12}(u)du \le d_4^{\frac12}\l\|\hat{a}\r\|_2^2 T^{-2}\int\limits_0^T\,\sqrt{T-u}\,du=\frac23 d_4^{\frac12}\l\|\hat{a}\r\|_2^2T^{-\frac12}.
   \label{inq4int_K_1^1/2_app}
  \end{equation}
 Obviously,
  \begin{equation}
     T^{-2}\int\limits_0^T\,K_2^{\frac12}(u)du =T^{-2}\int\limits_0^T\,(T-u)|B(u)|du \le3^{-\frac12}\|B\|_2T^{-\frac12},
   \label{inq4int_K_2^1/2_app}
  \end{equation}
  \begin{equation}
     T^{-2}\int\limits_0^T\,K_3^{\frac12}(u)du \le\frac23\|B\|_2T^{-\frac12},
   \label{inq4int_K_3^1/2_app}
  \end{equation}
  \begin{equation}
     T^{-2}\int\limits_0^T\,K_4^{\frac12}(u)du =T^{-2}\int\limits_0^T\,\l(\frac12\int\limits_0^{T-u}\ \int\limits_0^{T-u}\,\l(B^2(t-s+u)+B^2(t-s-u)\r)dtds\r)^{\frac12}du\le\frac23\|B\|_2T^{-\frac12}.
   \label{inq4int_K_4^1/2_app}
  \end{equation}

 From inequalities \eqref{inq4Exi^2T_app} - \eqref{inq4int_K_4^1/2_app} it follows
  \[
    \ExpV\xi^2(T)=O\l(T^{-\frac12}\r),\ \ \text{as}\ \ T\to\infty.
  \]
\end{proof}

The result of the lemma can be strengthened to a.s. convergence in \eqref{cnv4supT^-1int_e_eps_app}. Note also that in the proof we did not use the condition $\hat{a}\in L_1(\mathbb{R})$.

\section{LSE asymptotic normality}\label{app_LSE_AsymNorm}

$\indent$Cumbersome sets of conditions on the behavior of the nonlinear regression function are used in the proofs of the LSE asymptotic normality of the model parameter can be found, for example, in~\cite{IvLeo_SAoRF_En,Iv_AToNR,IvLeRuMeZhu_EoHCiRwCDE}, and it does not make sense to write here all of them. We will comment only on the conditions associated with the proof of the CLT for one weighted integral of the linear process $\varepsilon$ in the observation model \eqref{c_n_reg_m}.

Consider the family of the matrix-valued measures $\mu _T (dx;\,\alpha)=\l(\mu_T^{jl}(dx;\,\alpha)\r)_{j,l=1}^q$, $T>T_0$, $\alpha\in\mathcal{A}$, with densities
 \begin{equation}
    \mu_T^{jl}(x;\,\alpha)=g_T^j(x,\,\alpha)\overline{g_T^l(x,\,\alpha)} \l(\int\limits_{\mathbb{R}}\l|g_T^j(x,\,\alpha)\r|^2dx \int\limits_{\mathbb{R}}\l|g_T^l(x,\,\alpha)\r|^2dx\r)^{-\frac12},\ x\in\mathbb{R},
  \label{pre_SpMs4ANoLSE_app}
 \end{equation}
where
 \[
   g_T^j(x,\,\theta)=\int\limits_0^T\;e^{ixt}g_j(t,\,\theta)dt,\ j=\overline{1,q}.
 \]

\textbf{1)} Suppose that the weak convergence $\mu_T\ \Rightarrow\ \mu$ as $T\to\infty$ holds, where $\mu_T$ is defined by \eqref{pre_SpMs4ANoLSE_app} and $\mu$ is a positive definite matrix measure.

This condition means that the element $\mu^{jl}$ of the matrix-valued measure $\mu$ are complex measures of bounded variation, and the matrix $\mu(A)$ is non-negative definite for any set $A\in\mathcal{Z}$, with $\mathcal{Z}$ denoting the $\sigma$-algebra of Lebesgue measurable subsets of $\mathbb{R}$, and $\mu(\mathbb{R})$ is positive definite matrix, (see, for example, Ibragimov and Rozanov~\cite{IbrRoz_GRP}).

The following definition can be found in Grenander~\cite{Gren_OEoRCicoAD}, Grenander and Rosenblatt~\cite{GrenRos_SAoSTS}, Ibragimov and Rozanov~\cite{IbrRoz_GRP}, Ivanov and Leonenko~\cite{IvLeo_SAoRF_En}.

\begin{definition}
 The positive-definite matrix-valued measure $\mu(dx;\,\alpha)=\l(\mu^{jl}(x;\,\alpha)\r)_{j,l=1}^q$ is said to be the spectral measure of regression function $g(t,\,\alpha)$.
\end{definition}

Practically the components $\mu^{jl}(x;\,\alpha)$ are determined from the relations
 \begin{equation}
    R_{jl}(h;\,\alpha)=\lim\limits_{T\to\infty}\,d_{jT}^{-1}(\alpha)d_{lT}^{-1}(\alpha) \int\limits_0^T\,g_j(t+h,\,\alpha)g_l(t,\,\alpha)dx =\int\limits_{\mathbb{R}}e^{i\lambda h}\mu^{jl}(d\lambda;\,\alpha),\ j,l=\overline{1,q},
  \label{def_R_jl_app}
 \end{equation}
where it is supposed that the matrix function $\bigl(R_{jl}(h;\,\alpha)\bigr)$ is continuous at $h=0$.

Continuing Example \ref{exmp_trig_reg_fn} with the trigonometric regression function \eqref{trig_reg_fn_app} from Appendix \ref{app_LSE_cons}, we can state using \eqref{def_R_jl_app} that the function $g(t,\,\alpha)$ has a block-diagonal spectral measure $\mu(d\lambda;\,\alpha)$ (see e.g., Ivanov et al~\cite{IvLeRuMeZhu_EoHCiRwCDE}) with blocks
 \begin{equation}
    \l(\begin{array}{ccc}
      \varkappa_k & i\rho_k & \overline{\beta}_k \\
      -i\rho_k & \varkappa_k & \overline{\gamma}_k \\
      \beta_k & \gamma_k & \varkappa_k
    \end{array}\r)
    ,\ k=\overline{1,N},
  \label{blck_mtrx4spec_msr_ex-app}
 \end{equation}
where
 \[
   \beta_k=\frac{\sqrt{3}}{2C_k}\bigl(B_k\varkappa_k+iA_k\rho_k\bigr),\ \ \ \gamma_k=\frac{\sqrt{3}}{2C_k}\bigl(-A_k\varkappa_k+iB_k\rho_k\bigr),\ \ \ C_k=\sqrt{A_k^2+B_k^2},\ \ k=\overline{1,N}.
 \]
In \eqref{blck_mtrx4spec_msr_ex-app} the measure $\varkappa_k=\varkappa_k(d\lambda)$ and the signed measure $\rho_k=\rho_k(d\lambda)$ are concentrated at the points $\pm\varphi_k$, and $\varkappa_k\bigl(\l\{\pm\varphi_k\r\}\bigr)=\frac12$, $\rho_k\bigl(\l\{\pm\varphi_k\r\}\bigr)=\pm\frac12$.

Returning to the general case let the parameter $\alpha\in\mathcal{A}$ of regression function $g(t,\,\alpha)$ be fixed. We will use the notation $d_{iT}^{-1}(\alpha)g_i(t,\,\alpha)=b_{iT}(t,\,\alpha)$ and condition

\textbf{2)} $\sup\limits_{t\in[0,\,T]}\,\l|b_{iT}(t,\,\alpha)\r|\le c_iT^{-\frac12}$, $i=\overline{1,q}$.

The next CLT is an important part of the proof of LSE $\widehat{\alpha}_T$ asymptotic normality in the model \eqref{c_n_reg_m} and fully uses condition \textbf{A$_1$}.

\begin{theorem} \label{thm_AN4fncnl_appB}
  Under conditions \textbf{A$_1$, 1)} and \textbf{2)} the vector
   \begin{equation}
      \zeta_T=d_T^{-1}(\alpha)\int\limits_0^T\,\varepsilon(t)\nabla g(t,\,\alpha)dt =\l(\int\limits_0^T\,\varepsilon(t)b_{iT}(t,\,\alpha)dt\r)_{i=1}^q
    \label{rv4CLT_asym_norm_app}
   \end{equation}
  is asymptotically, as $T\to\infty$, normal $N(0,\Sigma)$,
   \[
     \Sigma=2\pi\int\limits_{\mathbb{R}}\,f(\lambda)\mu(d\lambda;\,\alpha) =d_2\int\limits_{-\infty}^\infty\,|a(\lambda)|^2\mu(d\lambda;\,\alpha).
   \]
\end{theorem}

\begin{prfapthm}
 For any $z=\l(z_1,\,\ldots,\,z_q\r)\in\mathbb{R}^q$ set
  \[
    \eta_T=\l<\zeta_T,\,z\r>=\int\limits_0^T\,\varepsilon(t)S_T(t)dt,\ \ S_T(t)=\sum\limits_{i=1}^q\,b_{iT}(t,\,\alpha)z_i.
  \]
 By condition \textbf{1)}
  \[
    \sigma^2(z)=\lim\limits_{T\to\infty}\,\ExpV\eta_T^2 =2\pi\int\limits_{\mathbb{R}}\,f(\lambda)\mu_z(d\lambda;\,\alpha),
  \]
 $\mu_z(d\lambda;\,\alpha)=\sum\limits_{i,j=1}^q\,\mu^{ij}(d\lambda;\,\alpha)z_iz_j$.

 To prove the theorem it is sufficient to show for any $z\in\mathbb{R}$  and $\nu\ge1$, that
  \begin{equation}
     \lim\limits_{T\to\infty}\,\ExpV\eta_T^n=\ExpV\eta^n
        =\l\{\begin{array}{ll}
              (n-1)!!\sigma^n(z), & n=2\nu, \\
                   0            , & n=2\nu+1.\end{array}\r.
   \label{lim4Eeta_T^n_app}
  \end{equation}

 Use the Leonov-Shiryaev formula (see, e.g., Ivanov and Leonenko~\cite{IvLeo_SAoRF_En}). Let
  \[
    I=\{1,\,2,\,\ldots,\,n\},\ I_p=\l\{i_1,\,\ldots,\,i_{l_p}\r\}\subset I,\ c(I_p)=c_{l_p}\l(t_{i_1},\,\ldots,\,t_{i_{l_p}}\r).
  \]
 Then
  \begin{equation}
      m(I)=m_n\l(t_1,\,\ldots,\,t_n\r)=\sum\limits_{A_r}\,\prod\limits_{p=1}^r\,c(I_p),
   \label{m_I_repr_app}
  \end{equation}
 where $\sum\limits_{A_r}$ denotes summation over all unordered partitions $A_r=\l\{\bigcup\limits_{p=1}^r\,I_p\r\}$ of the set $I$ into sets $I_1,\,\ldots,\,I_r$ such that $I=\bigcup\limits_{p=1}^r\,I_p$, $I_i\cap I_j=\emptyset$, $i\ne j$.

 Since
  \begin{equation}
     \ExpV\eta_T^n=\int\limits_{[0,\,T]^n}\,m_n(t_1,\,\ldots,\,t_n)\prod\limits_{k=1}^n\,R_T(t_k)dt_1\ldots dt_n,
   \label{Eeta_T^n_repr_app}
  \end{equation}
 then the application of formula \eqref{m_I_repr_app} to \eqref{Eeta_T^n_repr_app} shows that to obtain \eqref{lim4Eeta_T^n_app} it is sufficient to prove
  \begin{equation}
     I(l)=\int\limits_{[0,\,T]^l}\,c_l(t_1,\,\ldots,\,t_l)\prod\limits_{k=1}^l\,R_T(t_k)dt_1\ldots dt_l\ \ \longrightarrow\ 0,\ \ \text{as}\ \ T\to\infty.
   \label{cnv_I_l_2_0_app}
  \end{equation}
 for all $i=\overline{3,n}$. Taking into account the equality $\ExpV\varepsilon(t)=0$, from \eqref{cnv_I_l_2_0_app} will follow that in \eqref{lim4Eeta_T^n_app} all the odd moments $\ExpV\eta^{2\nu+1}=0$. On the other hand, for even moments $\ExpV\eta^{2\nu}$ we shall find that in \eqref{Eeta_T^n_repr_app} thanks to \eqref{m_I_repr_app} only those terms correspond to the partitions of the set $I=\{1,\,2,\,\ldots,\,2\nu\}$ into pairs of indices will remain nonzero, i.e. "Gaussian part" : all $l_p=2$. In \eqref{m_I_repr_app} it will be $(2\nu-1)!!$ of such terms and each of them will be equal to $\sigma^{2\nu}(z)$.

 Let us prove \eqref{cnv_I_l_2_0_app}. We note that condition \textbf{2)} implies
  \[
    \sup\limits_{t\in[0,\,T]}\,\l|R_T(t)\r|\le\|c\|\,\|z\|\,T^{-\frac12},\ \ c=\l(c_1,\,\ldots,\,c_q\r),\ \ z=\l(z_1,\,\ldots,\,z_q\r).
  \]
 Then using formula \eqref{rep_c_r_by_hat-a} we have
  \[
   \begin{aligned}
     |I(l)|=&\l|\,\int\limits_{[0,\,T]^l}\,c_l(t_1-t_l,\,\ldots,\,t_{l-1}-t_l,\,0)\prod\limits_{k=1}^l\,R_T(t_k)dt_1\ldots dt_l\r|\le\\
     \le&|d_l|\,\l|\int\limits_{\mathbb{R}}\,ds\int\limits_{[0,\,T]^l}\, \l(\prod\limits_{i=1}^{l-1}\,\hat{a}\l(t_i-t_l-s\r)\r)\hat{a}(-s)\prod\limits_{k=1}^l\,R_T(t_k)dt_1\ldots dt_l\r|\le\\
     \le&|d_l|\,\int\limits_{\mathbb{R}}\,\l|\hat{a}(-s)\r|\int\limits_0^T\, \l(\int\limits_0^T\,\l|\hat{a}\l(t-t_l-s\r)R_T(t)\r|dt\r)^{l-1} \l|R_T(t_l)\r|dt_lds\le\\
     \le&|d_l|\,\l(\|c\|^{l-1}\|z\|^{l-1}\l\|\hat{a}\r\|_1^lT^{-\frac{l-1}2}\r)\l(\|c\|\,\|z\|\,T^{\frac12}\r)=
   \end{aligned}
  \]
  \begin{equation}
     =|d_l|\,\bigl(\|c\|\,\|z\|\,\l\|\hat{a}\r\|_1\bigr)^lT^{-\l(\frac l2-1\r)}\ \to\ 0,\ \ \text{as}\ \ T\to\infty,\ \ l\ge3.
   \label{prf_cnv4I_l_2_0_app}
  \end{equation}
\end{prfapthm}

To obtain \eqref{prf_cnv4I_l_2_0_app} we have used $\hat{a}\in L_1(\mathbb{R})$ only.

Using the theorem, just as in the works cited above (for definiteness, we turn our attention to Ivanov et al~\cite{IvLeRuMeZhu_EoHCiRwCDE}), it can be proved that, if a number of additional conditions on the regression function are satisfied, the normalized LSE $d_T(\alpha_0)\l(\widehat{\alpha}_T-\alpha_0\r)$ is asymptotically normal $N\l(0,\,\Sigma_{_{LSE}}\r)$, with
 \[
    \begin{aligned}
       \Sigma_{_{LSE}}=&2\pi\l(\int\limits_{\mathbb{R}}\,\mu(d\lambda;\,\alpha_0)\r)^{-1} \int\limits_{\mathbb{R}}\,f(\lambda)\mu(d\lambda;\,\alpha_0)\, \l(\int\limits_{\mathbb{R}}\,\mu(d\lambda;\,\alpha_0)\r)^{-1}=\\
       &d_2\l(\int\limits_{\mathbb{R}}\,\mu(d\lambda;\,\alpha_0)\r)^{-1} \int\limits_{\mathbb{R}}\,|a(\lambda)|^2\mu(d\lambda;\,\alpha_0)\, \l(\int\limits_{\mathbb{R}}\,\mu(d\lambda;\,\alpha_0)\r)^{-1}.
    \end{aligned}
 \]

Note that, firstly, our conditions \textbf{N$_3$, 1), 2)} are included in the conditions for the LSE asymptotic normality of Ivanov et al~\cite{IvLeRuMeZhu_EoHCiRwCDE}, and, secondly, the trigonometric regression function \eqref{trig_reg_fn_app} satisfies the conditions of Ivanov et al~\cite{IvLeRuMeZhu_EoHCiRwCDE}. Moreover, using \eqref{blck_mtrx4spec_msr_ex-app} and \eqref{cnv4T^-1d_T_ex-app} we conclude that for the trigonometric model the normalized LSE
 \[
   \begin{aligned}
       &\Bigl(T^{\frac12}\l(A_{1T}-A_1^0\r),\,T^{\frac12}\l(B_{1T}-B_1^0\r),\,T^{\frac32}\l(\varphi_{1T}-\varphi_1^0\r), \,\ldots,\Bigr.\\
       &\hspace*{39mm} \Bigl.T^{\frac12}\l(A_{NT}-A_N^0\r),\,T^{\frac12}\l(B_{NT}-B_N^0\r),\, T^{\frac32}\l(\varphi_{NT}-\varphi_N^0\r) \Bigr)
   \end{aligned}
 \]
is asymptotically normal $N\l(0,\,\Sigma_{_{TRIG}}\r)$, where $\Sigma_{_{TRIG}}$ is a block diagonal matrix with blocks
 \[
   \dfrac{4\pi f\l(\varphi_k^0\r)}{\l(C_k^0\r)^2}
       \l(\begin{array}[pos]{ccc}
            \l(A_k^0\r)^2+4\l(B_k^0\r)^2 & -3A_k^0B_k^0                 & -6B_k^0 \\
                      -3A_k^0B_k^0       & \l(B_k^0\r)^2+4\l(A_k^0\r)^2 &  6A_k^0 \\
                        -6B_k^0          &          6A_k^0              & 12      \end{array}\r),\ k=\overline{1,N}.
 \]
The matrix $\Sigma_{_{TRIG}}$ is positive definite, if $f\l(\varphi_k^0\r)>0$, $k=\overline{1,N}$. Hovewer it follows from our condition \textbf{A$_2$(iii)}.

Note also that condition \textbf{N$_2$} is satisfied, for example, for the trigonometric regression function \eqref{trig_reg_fn_app}. Indeed, in this case
 \[
   g'(t,\,\alpha)=\sum\limits_{i=1}^N\,\l(-\varphi_iA_i\sin\varphi_it+\varphi_iB_i\cos\varphi_it\r),
 \]
and similarly to \eqref{inq4dif_trig_comp_ex-app}
 \[
   \begin{aligned}
    &\l|-\varphi_iA_i\sin\varphi_it+\varphi_iB_i\cos\varphi_it+\varphi_i^0A_i^0\sin\varphi_i^0t- \varphi_i^0B_i^0\cos\varphi_i^0t\r|\le\\
    &\hspace*{21mm}\le\overline{\varphi}\bigl(\l|A_i-A_i^0\r|+\l|B_i-B_i^0\r|\bigr) +\l(|A_i^0|+|B_i^0|\r)\l(1+\overline{\varphi}t\r)\l|\varphi_i-\varphi_i^0\r|,\ i=\overline{1,N},
   \end{aligned}
 \]
which leads to the inequality of condition \textbf{N$_2$} similar to \eqref{inq4Phi_T_l_C2_ex-app}, but with a different constant $c_0'$.

\section{Levitan polynomials}\label{app_Levitan_polnml}

$\indent$Some necessary facts of approximation theory adapted to needs of this article are represented in this Appendix. All the definitions and results are taken from the book \cite{Ahi_LoAT}.

In complex analysis entire function of exponential type is said to be such a function $F(z)$ that for any complex $z$ the inequality
 \begin{equation}
     F(z)\le A e^{B|z|}
   \label{ent_fn_o_exp_type}
 \end{equation}
holds true, where the numbers $A$ and $B$ do not depend on $z$. Infinum $\sigma$ of the constant $B$ values for which inequality \eqref{ent_fn_o_exp_type} takes place is called the exponential type of function $F(z)$ and can be determined by formula
 \[
   \sigma=\underset{|z|\to\infty}{\lim\sup}\,\dfrac{\ln|F(z)|}{|z|}.
 \]
Denote by $\mathcal{B}_\sigma$ the totality of all the entire functions $F(z)$ of exponential type $\le\sigma$ with property $\underset{\lambda\in\mathbb{R}}\sup\,|F(\lambda)|<\infty$.

Let $\mathcal{C}$ be linear normed space of bounded continuous functions $\varphi(\lambda)$, $\lambda\in\mathbb{R}$, with norm $\|\varphi\|=\underset{\lambda\in\mathbb{R}}\sup\,|\varphi(z)|<\infty$. Consider further some set of functions $\mathfrak{M}\subset \mathcal{C}$. For the function of interest $\varphi\in\mathfrak{M}$ suppose that
 \begin{equation}
    \underset{\eta\to1}\lim\,\underset{\lambda\in\mathbb{R}}\sup\,|\varphi(\eta\lambda)-\varphi(\lambda)|=0,
  \label{limsup_dif_phis_2_0}
 \end{equation}
and write
 \[
   \mathcal{A}_\sigma[\varphi]=\underset{F\in \mathcal{B}_\sigma}\inf\,\|\varphi-F\|.
 \]

Let $h(\lambda)$, $\lambda\in\mathbb{R}$, be uniformly continuous function. Denote by
 \[
   \omega(\delta)=\omega(\delta;\,h)=\underset{|\lambda_1-\lambda_2|\le\delta}\sup\,\l|h(\lambda_1)-h(\lambda_2)\r|,\ \lambda_1,\lambda_2\in\mathbb{R},\ \delta>0,
 \]
the modulus of continuity of the function $h$. Obviously $\omega(\delta),\ \delta>0$, is nondecreasing continuous function  tending to zero, as $\delta\to0$.

Let the set $\mathfrak{M}$ introduced above consists of differentiable functions such that for $\varphi\in\mathfrak{M}$ the derivatives $\varphi'(\lambda)=h(\lambda),\ \lambda\in\mathbb{R}$, are uniformly continuous on $\mathbb{R}$. Then for function $\varphi$ satisfying the property \eqref{limsup_dif_phis_2_0} there exists a function $F_\sigma\subset\mathcal{B}_\sigma$ such that (see \cite{Ahi_LoAT}, p. 252)
 \begin{equation}
    \mathcal{A}_\sigma[\varphi]=\|\varphi-F_\sigma\|\le\frac3\sigma\omega\l(\frac1\sigma;\,h\r).
  \label{A_sigma-phi_from_M}
 \end{equation}
The inequality \eqref{A_sigma-phi_from_M} means that for the described function $\varphi$ and any $\delta>0$ there exists a number $\sigma=\sigma(\delta)$ and a function $F_\sigma\in\mathcal{B}_\sigma$ such that
 \[
   \|\varphi-F_\sigma\|<\delta.
 \]

As it has been proved in the 40s of the 20th century by B.M.~Levitan for any function $F\in\mathcal{B}_\sigma$ it is possible to build a sequence of trigonometric sums $T_n(F;\,z),\ n\ge1$, bounded on $\mathbb{R}$ by the same constant as the function $F$, that converges to $F(z)$ uniformly in any bounded part of the complex plane. In particular, for any compact set $K\subset\mathbb{R}$
 \[
    \underset{n\to\infty}\lim\,\underset{\lambda\in K}\sup\,\l|F(\lambda)-T_n(F;\,\lambda)\r|=0.
 \]
Put $s=\frac{\sigma}n$, $n\in\mathbb{N}$; $c_j^{(n)}=sE_s(js)$, $j\in\overline{-n,n}$;
 \[
   E_s(x)=(2\pi)^{-1}\int\limits_{\mathbb{R}}\,e^{-\mathrm{i}xu}\l(\frac{2\sin\frac{su}2}{su}\r)^2F(u)du,\ x\in\mathbb{R}.
 \]
Then the sequence of the Levitan polynomials that corresponds to $F$ can be written as
 \[
    T_n(F;z)=\sum\limits_{j=-n}^n\,c_j^{(n)}e^{\mathrm{i}jsz}.
 \]
\end{appendices}
%----------------------------------------------

\renewcommand{\refname}{References}

\end{document}